\documentclass[12pt, reqno]{amsart}

\usepackage{mathtools}
\mathtoolsset{showonlyrefs}
\numberwithin{equation}{section}

\usepackage{graphicx}
\usepackage{hyperref}

\usepackage{mathtools}
\mathtoolsset{showonlyrefs}

\usepackage{amssymb}
\usepackage{amsthm}
\usepackage{amscd}
\usepackage{amsmath}
\usepackage{epic,eepic}
\usepackage{mathrsfs}
\usepackage{latexsym}
\usepackage{pifont}
\usepackage[all]{xy}
\usepackage {cancel}

\theoremstyle{plain}
\newtheorem{lemma}{Lemma}[section]
\newtheorem{prop}[lemma]{Proposition}
\newtheorem{thm}[lemma]{Theorem}
\newtheorem{cor}[lemma]{Corollary}
\newtheorem{intthm}{Theorem}

\theoremstyle{definition}

\newtheorem{rem}[lemma]{Remark}
\newtheorem{defi}[lemma]{Definition}
\newtheorem{exa}[lemma]{Example}

\newtheorem{problem}{Problem}

\newcommand{\bde}{\begin{defi}}
\newcommand{\ede}{\end{defi}\vspace{1mm}}
\newcommand{\ble}{\begin{lemma}}
\newcommand{\ele}{\end{lemma}}
\newcommand{\bpr}{\begin{prop}}
\newcommand{\epr}{\end{prop}}
\newcommand{\bt}{\begin{thm}}
\newcommand{\et}{\end{thm}}
\newcommand{\bco}{\begin{cor}}
\newcommand{\eco}{\end{cor}}
\newcommand{\bre}{\begin{rem}}
\newcommand{\ere}{\end{rem}}
\newcommand{\bex}{\begin{exa}}
\newcommand{\eex}{\end{exa}}
\newcommand{\bpf}{\begin{proof}}
\newcommand{\epf}{\end{proof}}

\newcommand{\mcA}{\mathcal{A}}
\newcommand{\mcB}{\mathcal{B}}
\newcommand{\mcC}{\mathcal{C}}
\newcommand{\mcD}{\mathcal{D}}
\newcommand{\mcE}{\mathcal{E}}
\newcommand{\mcF}{\mathcal{F}}
\newcommand{\mcG}{\mathcal{G}}
\newcommand{\mcH}{\mathcal{H}}

\newcommand{\mcL}{\mathcal{L}}
\newcommand{\mcM}{\mathcal{M}}
\newcommand{\mcN}{\mathcal{N}}
\newcommand{\mcO}{\mathcal{O}}

\newcommand{\mcS}{\mathcal{S}}
\newcommand{\mcT}{\mathcal{T}}

\newcommand{\mcW}{\mathcal{W}}

\newcommand{\mbC}{\mathbb{C}}

\newcommand{\mbF}{\mathbb{F}}
\newcommand{\mbG}{\mathbb{G}}

\newcommand{\mbY}{\mathbb{Y}}
\newcommand{\mbZ}{\mathbb{Z}}

\newcommand{\mfS}{\mathfrak{S}}

\newcommand{\mfb}{\mathfrak{b}}

\newcommand{\mfg}{\mathfrak{g}}
\newcommand{\mfh}{\mathfrak{h}}

\newcommand{\mfl}{\mathfrak{l}}

\newcommand{\mfo}{\mathfrak{o}}
\newcommand{\mfp}{\mathfrak{p}}

\newcommand{\mfs}{\mathfrak{s}}
\newcommand{\mft}{\mathfrak{t}}

\newcommand{\msA}{\mathscr{A}}

\newcommand{\msC}{\mathscr{C}}

\newcommand{\msE}{\mathscr{E}}
\newcommand{\msF}{\mathscr{F}}

\newcommand{\msL}{\mathscr{L}}

\newcommand{\msN}{\mathscr{N}}

\newcommand{\msT}{\mathscr{T}}

\newcommand{\msX}{\mathscr{X}}

\newcommand{\Dual}{\rotatebox[origin=c]{180}{$D$}\hspace{-0.5mm}}

\newcommand{\X}{{}}

\newcommand{\BB}{\blacktriangledown}
\newcommand{\BBB}{\blacktriangle}
\newcommand{\BBBB}{\triangledown}

\newcommand{\ang}{\sphericalangle}
\newcommand{\SSP}{\vspace{3mm}}

\newcommand{\mr}{\mathrm}
\newcommand{\bb}{\vartheta}
\newcommand{\Fus}{\rotatebox[origin=c]{180}{$\mbY$}}

\begin{document}

\title[Duality for dormant opers of classical types B and  C]{Duality for dormant opers of classical types B and  C}
\author{Yasuhiro Wakabayashi}
\date{\today}
\markboth{}{}
\maketitle
\footnotetext{Y. Wakabayashi: 
Graduate School of Information Science and Technology, The University of Osaka, Suita, Osaka 565-0871, Japan;}
\footnotetext{e-mail: {\tt wakabayashi@ist.osaka-u.ac.jp};}
\begin{abstract} 
A $\mathfrak{g}$-oper for a  simple Lie algebra $\mathfrak{g}$ is a specific type of flat principal bundle on an algebraic curve.
When the base field is of prime characteristic $p$, those with vanishing $p$-curvature are  called dormant $\mathfrak{g}$-opers, and they  form finite and geometrically meaningful moduli spaces. In earlier work, a canonical duality was  established  between dormant $\mathfrak{sl}_n$-opers  and dormant $\mathfrak{sl}_{p-n}$-opers. This duality has  provided effective tools for the study of higher-rank cases, as well as for the computation and structural understanding of the associated enumerative invariants. The main result of this paper extends this duality phenomenon to classical Lie algebras of type B and C. More precisely, under the numerical condition $p-1 = 2 (\ell +m)$, we construct a canonical isomorphism between the moduli spaces of dormant $\mathfrak{so}_{2\ell +1}$-opers and dormant  and $\mathfrak{sp}_{2m}$-opers with prescribed symmetric radii.

\end{abstract}
\vspace{10mm}
\tableofcontents 

\section{Introduction} \label{S0}

One of the fundamental invariants for  linear differential equations, and more generally flat bundles,
in prime characteristic $p>0$ is  {\it $p$-curvature}.
Roughly speaking, the $p$-curvature measures the failure of such a structure  to be compatible with $p$-power operations 
 on certain associated spaces of infinitesimal symmetries.
It is well-known  that 
a linear  homogeneous differential equations has vanishing $p$-curvature if and only if its  space of  solutions attains the  maximal possible rank.

Flat bundles with vanishing $p$-curvature
have played a central role in several areas of arithmetic and algebraic geometry.
Most notably,
 they appear in the study for the Grothendieck-Katz conjecture, which predicts a deep relationship between the vanishing of 
$p$-curvature for almost all primes and the algebraicity of solutions of complex differential equations  (cf. ~\cite{Kat2}, ~\cite{And}).
 Moreover, characteristic-$p$
analogues of non-abelian Hodge theory and of the geometric Langlands correspondence have been developed via a variant  of  the Hitchin morphism incorporating   
$p$-curvature.
This has led  to a rich theory connecting flat bundles and Higgs bundles through the Cartier transform (cf.  ~\cite{BeBr}, ~\cite{ChZh1}, ~\cite{ChZh2}, ~\cite{GLQ}, ~\cite{LSZ}, ~\cite{OgVo}, and ~\cite{She}).

In this paper, we focus on a particular class of flat bundles, namely
 {\it dormant $\mfg$-opers} for a simple Lie algebra $\mfg$  (cf. Section \ref{SS02} or  ~\cite[Definitions 2.1 and  3.15]{Wak5} for the definition of a dormant $\mfg$-oper).
 Dormant $\mfg$-opers
   may be viewed as natural generalizations of  linear  homogeneous ordinary differential operators with unit principal symbol and
   vanishing $p$-curvature.
 
 Hereinafter, suppose that $p$ does not divide the order of the Weyl group of $\mfg$. 
 Fix  a pair of nonnegative integers $(g, r)$ satisfying  $2g-2 +r > 0$, and let
  $\overline{\mcM}_{g, r}$ denote  the moduli stack of $r$-pointed stable curves of genus $g$  over a fixed algebraically closed field $k$ of  characteristic $p$.
Given a  collection of certain additional data $\rho := (\rho_i)_{i=1}^r$ specified   at the marked points, 
one can consider  the moduli stack
\begin{align} \label{eeRt2}
\mcO p_{\mfg, \rho, g, r}^{^\mr{Zzz...}}
\end{align}
(cf. ~\cite[Eq.\,(433)]{Wak5}) that
 parametrizes   pairs $(\msX, \msE^\spadesuit)$ consisting of a pointed stable curve   $\msX$ in $\overline{\mcM}_{g, r}$  and  a dormant $\mfg$-oper $\msE^\spadesuit$ of radii $\rho$ on it (cf. Sections \ref{SS600}-\ref{SS6002} or ~\cite[Definition 2.29]{Wak5} for the definition of radius).
  For $\mfg = \mfs \mfl_2$, this stack was originally introduced in the context of $p$-adic Teichm\"{u}ller theory (cf. ~\cite{Moc2}), where   
 dormant $\mfs \mfl_2$-opers (or more generally, certain $\mfs \mfl_2$-opers with nilpotent $p$-curvature)
play a role  analogous  to  ``nice" projective structures on Riemann surfaces, such as those  arising   from uniformization.

 \hspace{5mm} 
 \includegraphics[width=16cm,bb=0 0 950 180,clip]{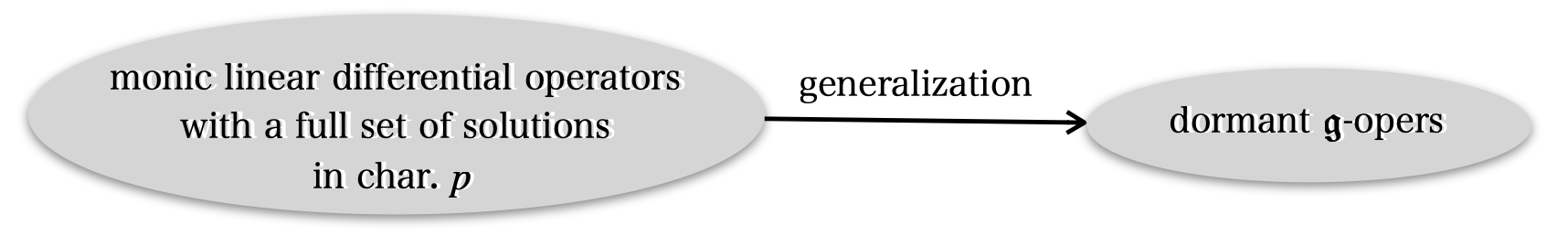}

\vspace{5mm}
 
 For a general simple Lie algebra $\mfg$ satisfying certain assumptions,
it is known  that $\mcO p^{^\mr{Zzz...}}_{\mfg, \rho, g, r}$
 is a  (possibly empty) proper Deligne-Mumford  stack, and that 
 the projection $\Pi_{\mfg, \rho, g, r} : \mcO p^{^\mr{Zzz...}}_{\mfg, \rho, g, r}\rightarrow \overline{\mcM}_{g, r},$ given by $(\msX, \msE^\spadesuit)\mapsto \msX$ is finite
 (cf. ~\cite[Theorem C]{Wak5}).
(Under some restricted situations, the finiteness  has been established earlier in ~\cite[Chap.\,II, Theorem 2.8]{Moc2},  ~\cite[Corollary 6.1.6]{JP}, and ~\cite[Theorem A.1]{BeTr}.)
Furthermore, in many cases, the morphism $\Pi_{\mfg, \rho, g, r}$ is {\it generically \'{e}tale}  (cf. ~\cite[Theorem G]{Wak5}, ~\cite[Theorem A]{Wak10}).

The 
generic \'{e}taleness is of crucial importance from the  viewpoint  of enumerative geometry;
it allows one to interpret the generic degree  $\mr{deg} (\Pi_{\mfg, \rho, g, r})$ of $\Pi_{\mfg, \rho, g, r}$ as well-defined enumerative invariants, 
and to describe its behavior under degeneration of underlying curves by means of a kind of  {\it fusion rule}
 or a {\it two-dimensional topological quantum field theory ($2$d TQFT)}.
In the case $\mfg = \mfs \mfl_n$, 
such a  description leads to explicit formulas for the number of dormant opers.
These formulas  can be compared with Verlinde-type formulas arising from the conformal field theory of affine Lie algebras,  with Gromov-Witten invariants of Grassmannians, and with special values of the Ehrhart quasi-polynomials associated with rational polytopes.
However, these approaches are effective mainly when the rank $n$ is relatively small compared  to the characteristic $p$. 
For further details, see ~\cite{Jo14}, ~\cite{JP}, ~\cite{LiOs}, ~\cite{Moc2}, ~\cite{Wak1},  ~\cite{Wak3}, ~\cite{Wak5}, and ~\cite{Wak20}.

A key breakthrough for higher-rank cases was the discovery of a {\it duality} between dormant 
 $\mfs \mfl_n$-opers and dormant $\mfs \mfl_{p-n}$-opers.
 According to ~\cite[Theorem A]{Wak2}, 
this duality yields, for each $r$-tuple $\rho := (\rho_i)_{i=1}^r$  of elements in  a certain finite set $\Xi_{\mfs \mfl_n}$ (cf. \eqref{EQ300}),  an isomorphism
\begin{align} \label{EQ448}
\Dual_{n, \rho, g, r} : \mcO p^{^\mr{Zzz...}}_{\mfs \mfl_n, \rho, g, r} \xrightarrow{\sim} \mcO p^{^\mr{Zzz...}}_{\mfs \mfl_{p-n}, \rho^\BBBB, g, r}
\end{align}
(cf. \eqref{EQ604})  between the corresponding moduli stacks over  $\overline{\mcM}_{g, r}$, which satisfies the equality $\Dual_{p-n, \rho^\BBBB, g, r} \circ \Dual_{n, \rho, g, r} = \mr{id}$.
 Here,  $\rho^\BBBB$ denotes  the image of $\rho$ under a certain bijection  $(-)^\BBBB : \Xi_{\mfs \mfl_n} \xrightarrow{\sim} \Xi_{\mfs \mfl_{p-n}}$   (cf. \eqref{EQ304}).
 In particular, this duality  allows one to deduce results for large rank from those for smaller rank.
As a consequence, the associated 
fusion rules
 in  higher-rank cases can be described explicitly in terms of lower-rank data.

The purpose of the present paper is to establish an analogue of this duality for {\it classical Lie algebras of types B and C}.
More precisely, we prove a canonical duality between dormant 
$\mfs \mfo_{2\ell +1}$-opers and dormant 
$\mfs \mfo_{2m}$-opers under the numerical condition 
$p-1 =2 (\ell + m)$.
It  may be regarded as a natural extension of the {\it rank-level duality} phenomenon for 
$\mfs \mfl_n$, reflecting a  symmetry between orthogonal and symplectic structures in characteristic $p$;
 the Lie algebra is transformed  into its {\it Langlands dual} with its rank reversed at $p$ in a certain sense.

Our main result can be summarized as follows.

\begin{intthm}[cf. Theorem \ref{QW1199} for the full statement]\label{ThmA}
\begin{itemize}
\item[(i)]
Let  $\ell$, $m$ be positive integers with $p -1 = 2 (\ell + m)$, and let
 $\rho := (\rho_i)_{i=1}^r$ be an $r$-tuple of elements in $\Xi_{\mfs \mfo_{2\ell +1}}$ (cf. \eqref{EQ300}),
 where we set $\rho := \emptyset$ if  $r = 0$.
Then,  there exists 
a duality isomorphism 
 \begin{align} \label{EQ40}
 \Dual_{2\ell +1, \rho, g, r}^\ang : \mcO p_{\mfs \mfo_{2\ell +1} \rho, g, r}^{^\mr{Zzz...}} \xrightarrow{\sim} 
 \mcO p_{\mfs \mfp_{2m}, \rho^\BBBB, g, r}^{^\mr{Zzz...}}
 \end{align}
over $\overline{\mcM}_{g, r}$ (cf. \eqref{EQ28e} for the definition of $\rho^\BBBB$).
 \item[(ii)]
 The stack $\mcO p_{\mfs \mfp_{p-1}, \rho, g, r}^{^\mr{Zzz...}}$ is empty unless $\rho = \rho_\mr{full}^r : = (\rho_\mr{full}, \cdots, \rho_\mr{full})$ (cf. \eqref{EQ301} for the definition of $\rho_\mr{full}$).
 In this exceptional case, the projection $\Pi_{\mfs \mfp_{p-1}, \rho^r_\mr{full}, g, r}$ defines an isomorphism 
 \begin{align}
  \mcO p_{\mfs \mfp_{p-1}, \rho^r_\mr{full}, g, r}^{^\mr{Zzz...}} \xrightarrow{\sim} \overline{\mcM}_{g, r}.
 \end{align}
 \end{itemize}
 \end{intthm}

 The resulting  isomorphism  \eqref{EQ40} makes it possible to achieve  a detailed understanding of 
the moduli spaces of 
dormant $\mfs \mfo_{2\ell +1}$- and $\mfs \mfp_{2m}$-opers
for large rank by reducing the problem to  the corresponding  cases of smaller rank, where detailed results are already available  (cf. ~\cite[Theorems F and G]{Wak5}).
As an application,
we construct  
a certain kind of  fusion rule, called a {\it pseudo-fusion rule}
\footnote{These factorization properties  together  form  a fusion rule.
However, due to the lack of appropriate references, we will describe them here  using  the notion of a pseudo-fusion rule, following the framework established in the primary reference ~\cite{Wak5}.}, 
   governing factorization properties on enumerative invariants of these moduli  spaces.
These structures   yield
 Verlinde-type formulas that extend the previously known results to higher-rank cases.
The precise statements  are as follows.

\begin{intthm}[cf. Theorems \ref{Cor21} and \ref{Thm15}] \label{ThmB}
Suppose that  $\mfg = \mfs \mfo_{2\ell +1}$ (resp., $\mfg = \mfs \mfp_{2m}$) for some positive integer  $\ell$ (resp., $m$) with  $1 \leq \ell \leq \frac{p-3}{2}$ (resp., $1 \leq m \leq \frac{p-3}{2}$).
\begin{itemize}
\item[(i)]
Let $\rho := (\rho_{i})_{i=1}^r$ be an $r$-tuple of  elements of $\Xi_\mfg$, where we set $\rho := \emptyset$ when $r = 0$.
Then,   the moduli stack $\mcO p_{\mfg, \rho, g, r}^{^\mr{Zzz...}}$ can be represented by a (possibly empty) proper Deligne-Mumford stack over $k$ and the projection $\Pi_{\mfg, \rho, g, r}$
 is finite and generically \'{e}tale.
\item[(ii)]
Denote by $\Fus_\mfg$ the pseudo-fusion ring for dormant $\mfg$-opers, equipped  with the multiplication $\ast$ (cf. ~\cite[Definition 7.34]{Wak5}).
 Write $\mfS$ for the set of ring homomorphims $\Fus_\mfg \rightarrow \mbC$ and write $\mr{Cas} := \sum_{\lambda \in \Xi_\mfg} \lambda \ast \lambda \left(\in \Fus_\mfg \right)$.
Then, for each $\rho := (\rho_i)_{i=1}^r \in \Xi_\mfg^{r}$,
the following equality holds:
\begin{align}
\mr{deg}(\Pi_{\mfg, \rho, g, r}) = \sum_{\chi \in \mfS} \chi (\mr{Cas})^{g-1} \cdot \prod_{i=1}^r \chi (\rho_i).
\end{align}
In particular, if $r = 0$ (which implies $g > 1$), then this equality reads
\begin{align}
\mr{deg}(\Pi_{\mfg, \emptyset, g, 0}) = \sum_{\chi \in \mfS} \chi (\mr{Cas})^{g-1}.
\end{align}
\end{itemize}
 \end{intthm}

Finally, we remark that, even in the cases $\mfg = \mfs \mfo_{2\ell +1}$ and $\mfg= \mfs \mfp_{2m}$, 
a detail analysis   of the stacks  
$\mcO p^{^\mr{Zzz...}}_{\mfg, \rho, 0, 3}$ is indispensable for the explicit determination of   
the degrees $\mr{deg}(\Pi_{\mfg, \rho, g, r})$, since 
   these values  characterize the ring-theoretic structure of the associated fusion ring $\Fus_\mfg$.
 Addressing this issue constitutes an important direction for future research in the enumerative geometry of dormant opers.

\hspace{5mm} 
 \includegraphics[width=16cm,bb=0 0 750 260,clip]{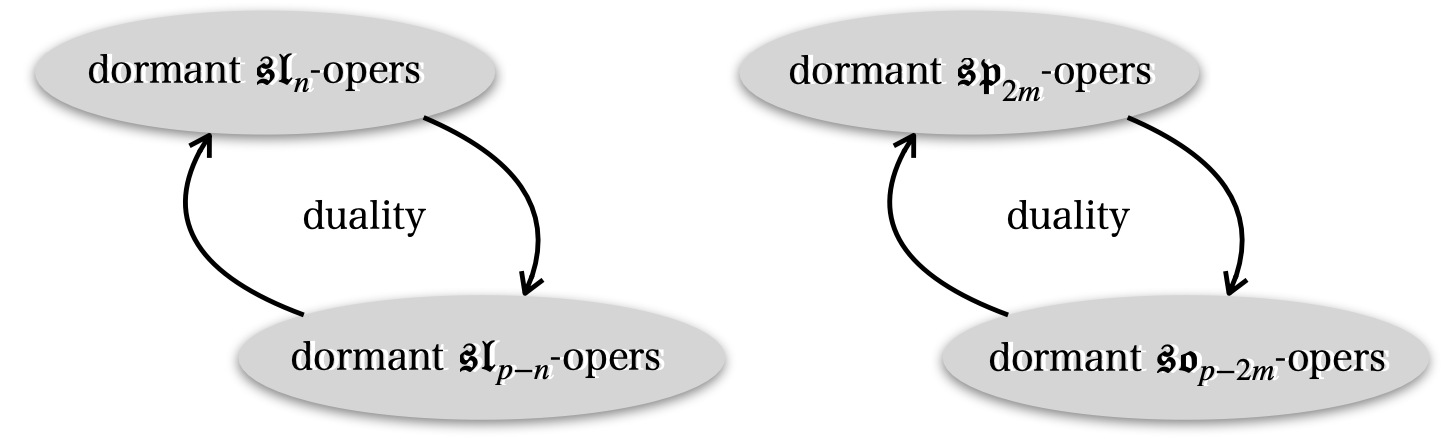}


\vspace{10mm}
\section{Opers on a pointed stable curve} \label{S1}

This section reviews some basic definitions and facts concerning 
dormant opers on a pointed stable curve, formulated within the framework of logarithmic geometry.
In particular, we recall from ~\cite{Wak5} the descriptions of dormant $\mfs \mfl_n$-, $\mfs \mfo_{2\ell +1}$-, and $\mfs \mfp_{2m}$-opers in terms of vector bundles.

Throughout  this paper, we fix a pair of nonnegative  integers $(g, r)$ with $2g-2+r>0$,  a prime number $p$,  and an algebraically closed field $k$ of characteristic $p$.
All schemes appearing in this paper are assumed to be locally noetherian.

\subsection{Algebraic groups and Lie algebras} \label{SS01}

Let us fix 
  a connected simple algebraic group $G$  over $k$ of adjoint type. 
  Also,  fix  a maximal torus $T$
  of $G$ and 
  a Borel subgroup $B$
    of $G$ containing $T$.
    Denote by $\mfg$, $\mfb$, and  $\mft$   the  Lie algebras of $G$, $B$,  and $T$, respectively (hence $\mft \subseteq \mfb \subseteq \mfg$).
In what follows, we assume that  {\it  the characteristic $p$ of the base field $k$ does not divide the order of   the Weyl group $W$ of $(G, T)$.}

Let
$\Phi^+$ be   the set of positive  roots in $B$ with respect to $T$ and by   $\Phi^-$ the set of negative roots. 
Also, denote by $\Gamma \left(\subseteq \Phi^+\right)$ the set of  simple positive  roots.
Given each character $\varphi : T \rightarrow \mbG_m$,
 we abuse notation by writing $\varphi$ for 
  its differential $d \varphi \in \mft^\vee$.
For each $\alpha \in \Phi^+ \cup \Phi^-$, we write
\begin{align} \label{22223}
\mfg^\alpha := \big\{ x \in \mfg \ \big| \ \text{$\mr{ad}(t)(x) =  \alpha  (t) \cdot x$ for all  $t\in T$}  \big\}.
\end{align}
Each $\mfg^{-\alpha}$ (for $\alpha \in \Phi^+$)
can  be regarded as a subspace of $\mfg / \mfb$ closed under the adjoint $B$-action.
The Lie algebra $\mfg$ is equipped with the principal grading  $\mfg = \bigoplus_{i = - \mr{rk}(\mfg)}^{\mr{rk}(\mfg)} \mfg_i$,  which
restricts to the identifications $\mft = \mfg_0$, $\bigoplus_{\alpha \in \Gamma} \mfg^\alpha = \mfg_1$,  and $\bigoplus_{\alpha \in \Gamma} \mfg^{-\alpha} = \mfg_{-1}$.


For  a positive integer $n$, we denote by $\mr{GL}_n$ (resp., $\mr{PGL}_n$) the general (resp., projective) linear group of rank  $n$. 
For a positive integer $\ell$ (resp., $m$), we denote  
 by $\mr{GO}_{2\ell+1}$
  (resp., $\mr{GSp}_{2m}$) 
  the group of  orthogonal similitudes  of rank $2\ell+1$  (resp., the group of symplectic similitudes  of rank  $2m$).
The Lie algebra of its adjoint group is given by 
\begin{align}
\mfs \mfo_{2\ell+1} := \left\{ v \in \mfs \mfl_{2\ell+1} \, | \, {^t}v K_{2\ell+1} + K_{2\ell+1}v = 0\right\}  \\ (\text{resp.,} \  \mfs \mfp_{2m} := \left\{ v \in \mfs \mfl_{2m} \, | \, {^t}v J_{2m} + J_{2m}v = 0\right\}),
\end{align}
where  $K_a := \begin{pmatrix} 0 & 0  &\cdots  & 0 & 1 \\  0 & 0 & \cdots & 1 & 0  \\ \vdots & \vdots &  \reflectbox{$\ddots$} & \vdots    & \vdots  \\  0  & 1 & \cdots  & 0 & 0 \\1 & 0  & \cdots & 0 & 0  \end{pmatrix}  \in \mr{GL}_a$
and 
$J_{2a} := \begin{pmatrix} 0 &  K_a \\ - K_a & 0\end{pmatrix}$ $\in \mr{GL}_{2a}$ for $a \in \mbZ_{>0}$.

Finally, for  a smooth algebraic group $G_0$ over $k$,
 a (right) $G_0$-bundle
 $\mcE$    on  a scheme  $Y$, and 
 a $k$-vector space $\mfh$ equipped with a (left) $G_0$-action,
 we shall write $\mfh_{\mcE}$ for the vector bundle on $Y$ associated with the relative affine space $\mcE \times^{G_0} \mfh \left(:= (\mcE \times_k \mfh) /G_0 \right)$ 
 over $Y$.

\subsection{Log curves and pointed stable curves} \label{SS1}

 For basic properties of log schemes (or more generally,  log stacks),
we refer the reader to  ~\cite{KaKa}, ~\cite{ILL}, and ~\cite{KaFu}.

Let $S^\mr{log}$ be  an fs log scheme over $k$.
Following  ~\cite[Definition 2.1]{Wak20}, we define
 a {\bf log curve} over  $S^\mr{log}$ to be 
 a log smooth integral morphism of fs log $k$-schemes  $f^\mr{log} : X^\mr{log} \rightarrow S^\mr{log}$ such that each geometric fiber of the underlying morphism of schemes $f : X \rightarrow S$ is a (possibly empty) reduced $1$-dimensional scheme.
In particular, both $\Omega_{X^\mr{log}/S^\mr{log}}$ and $\mcT_{X^\mr{log}/S^\mr{log}}$ are line bundles, and the underlying morphism $f : X \rightarrow S$ is flat (cf.  ~\cite[Corollary 4.5]{KaKa}).

Let  $\overline{\mcM}_{g, r}$ denote 
the moduli stack classifying $r$-pointed stable curves of genus $g$ over $k$, and let  $\mcM_{g, r} \left(\subseteq  \overline{\mcM}_{g, r} \right)$ its  dense open substack parametrizing   smooth curves.
We  denote by 
\begin{align}
\msC_{g, r} := (f_{\mr{univ}} : \mcC_{g, r} \rightarrow \overline{\mcM}_{g, r}, \{ \sigma_{\mr{univ}, i} : \overline{\mcM}_{g, r} \rightarrow \mcC_{g, r} \}_{i=1}^r)
\end{align}
the universal family of $r$-pointed stable curves over $\overline{\mcM}_{g, r}$,
   which consists of a prestable curve 
  $f_{\mr{univ}} : \mcC_{g, r} \rightarrow \overline{\mcM}_{g, r}$ over 
  $\overline{\mcM}_{g, r}$ and a collection of mutually disjoint  marked points
  $\sigma_{\mr{univ}, i} : \overline{\mcM}_{g, r} \rightarrow \mcC_{g, r}$ ($i =1, \cdots, r$).
Recall from ~\cite[Theorem 4.5]{KaFu} that  both $\overline{\mcM}_{g, r}$  and $\mcC_{g, r}$ admit natural log structures.
 We denote the resulting fs log stacks by $\overline{\mcM}_{g, r}^\mr{log}$  and $\mcC_{g, r}^\mr{log}$, respectively.
More precisely, the log structure of $\overline{\mcM}_{g, r}^\mr{log}$ is induced by 
the normal crossing divisor defined as the boundary $\partial \overline{\mcM}_{g, r} := \overline{\mcM}_{g, r}\setminus \mcM_{g, r}$.

Let $\msX := (f : X \rightarrow S, \{ \sigma_i \}_{i=1}^r)$ be an $r$-pointed stable curve of genus $g$ over a $k$-scheme $S$, where $\sigma_i$ denotes the $i$-th marked point $S \rightarrow X$.
Both $S$ and $X$ are equipped with log structures pulled-back from $\overline{\mcM}_{g, r}^\mr{log}$ and $\mcC_{g, r}^\mr{log}$, respectively,  via 
 the classifying morphism $S \rightarrow \overline{\mcM}_{g, r}$ of $\msX$.
We denote the resulting fs log schemes by $X^\mr{log}$ and $S^\mr{log}$, respectively.
 The structure morphism $f : X \rightarrow S$ extends to a morphism of log schemes $f^\mr{log} : X^\mr{log} \rightarrow S^\mr{log}$, forming a log  curve.

\subsection{Opers on a log curve} \label{SS02}

Let us take an fs log scheme $S^\mr{log}$ over $k$ and a log curve  $f^\mr{log} : X^\mr{log} \rightarrow S^\mr{log}$  over $S^\mr{log}$.
For each $j \in \mbZ_{\geq 0} \sqcup \{ \infty \}$,
we denote by $\mcD_{\leq j}$ the sheaf of logarithmic crystalline differential operators of order $\leq j$ on $X^\mr{log}/S^\mr{log}$, i.e., the sheaf ``$\mcD_{\hslash, Y^\mr{log}/T^\mr{log}}^{< j+1}$" defined in ~\cite[Section 4.2.1]{Wak5}, where  the pair $(Y^\mr{log}/T^\mr{log}, \hslash)$ is taken to be  $(X^\mr{log}/S^\mr{log}, 1)$.
For simplicity of notation, we write $\Omega := \Omega_{X^\mr{log}/S^\mr{log}}$, $\mcT := \mcT_{X^\mr{log}/S^\mr{log}}$, and $\mcD := \mcD_{\leq \infty}$.

Now, let us take a (right) $B$-bundle $\pi_B : \mcE_B \rightarrow X$ on $X$.
Denote by $\pi : \mcE \rightarrow X$ the $G$-bundle associated with $\mcE_B$ via extension of structure group along  the inclusion $B \hookrightarrow G$.
In particular, $\mcE_B$ determines a $B$-reduction of $\mcE$.
The relative affine space $\mcE \times^G \mfg$ associated with 
the adjoint $G$-action  on $\mfg$ defines 
 a vector bundle $\mfg_{\mcE}$ on $X$, i.e., the adjoint vector bundle of $\mcE$.

Suppose that we are given
an $S^\mr{log}$-connection $\nabla : \mcT \rightarrow (\pi_*(\mcT_{\mcE^\mr{log}/S^\mr{log}}))^G$  on   $\mcE$, where $\mcE^\mr{iog}$ denotes the scheme $\mcE$ equipped with the log structure pulled-back from $X^\mr{log}$ via $\pi$ and $(\pi_*(\mcT_{\mcE^\mr{log}/S^\mr{log}}))^G$ denotes the subsheaf of $\pi_*(\mcT_{\mcE^\mr{log}/S^\mr{log}})$ consisting of sections invariant under the natural $G$-action  (cf. ~\cite[Definition 1.17]{Wak5} for the definition of an $S^\mr{log}$-connection).
 Let us  choose, locally on $X$,  an $S^\mr{log}$-connection $\nabla'$ on $\mcE$ preserving $\mcE_B$ (i.e., $\mr{Im}(\nabla) \subseteq (\pi_{B*}(\mcT_{\mcE_B^\mr{log}/S^\mr{log}}))^B$), and take the difference $\nabla - \nabla'$, which specifies   a section of $\Omega \otimes \mfg_{\mcE_B} \left(=\Omega \otimes \mfg_{\mcE}\right)$.
The local section of  $\Omega\otimes (\mfg/\mfb)_{\mcE_B}$ determined by this section via projection
does not depend on the choice of  $\nabla'$.
Hence, these sections defined for various $\nabla'$'s 
can be glued together to yield   a global section of $\Omega \otimes (\mfg/\mfb)_{\mcE_B}$; we shall denote this section
 by $\nabla /\mcE_B$.


A {\bf $\mfg$-oper} on 
$X^\mr{log}/S^\mr{log}$ is defined to be 
 a pair
$\msE^\spadesuit := (\mcE_B, \nabla)$, where
$\mcE_B$ and  $\nabla$ are as above such that 
the section  $\nabla/ \mcE_B$ lies in the submodule $\Omega \otimes (\bigoplus_{\alpha \in \Gamma} \mfg^{-\alpha})_{\mcE_B}$
 and that  its  image in $\Omega  \otimes \mfg^{- \beta}_{\mcE_B}$ (for each $\beta \in \Gamma$) via the projection $\bigoplus_{\alpha \in \Gamma} \mfg^{-\alpha} \twoheadrightarrow \mfg^{-\beta}$ specifies 
a nowhere vanishing section.
In a natural manner, one can define the notion of an isomorphism between $\mfg$-opers.
When $X^\mr{log}/S^\mr{log}$ arises from a pointed stable  curve $\msX$ over the underlying scheme $S$ of $S^\mr{log}$,
any $\mfg$-oper on $X^\mr{log}/S^\mr{log}$ will be referred to as  a $\mfg$-oper {\it on $\msX$}.
Also,  a $\mfg$-oper is said to be {\bf dormant} if its $p$-curvature vanishes identically  (cf. ~\cite[Definition 3.8]{Wak5} for the definition of the $p$-curvature of a logarithmic  flat $G$-bundle).

Next, we recall  the definition of a $\mr{GL}_n$-oper, where $n$ is a positive integer,  formulated in terms of vector bundles.
Consider a collection  of data
\begin{align} \label{EQ53}
\msF^\heartsuit := (\mcF, \nabla, \{ \mcF^j \}_{j=0}^n),
\end{align}
where 
\begin{itemize}
\item
$\mcF$ denotes a vector bundle on $X$ of rank $n$;
\item
$\nabla$ denotes an $S^\mr{log}$-connection on $\mcF$ in the sense of ~\cite[Definition 4.1]{Wak5}, i.e.,
an $f^{-1} (\mcO_S)$-linear morphism $\nabla : \mcF \rightarrow \Omega \otimes \mcF$ satisfying the Leibniz rule: $\nabla (a \cdot v) = da \otimes v + a \cdot \nabla (v)$ for any local sections $a \in \mcO_X$ and $v \in \mcF$;
\item
$\{ \mcF^j \}_j$ is a complete flag   on $\mcF$, i.e., a chain of subbundles
\begin{align}
0 = \mcF^n \subseteq \mcF^{n-1} \subseteq \cdots \subseteq \mcF^0 = \mcF
\end{align}
of $\mcF$ such that the subquotients $\mcF^j /\mcF^{j+1}$ are line bundles.
\end{itemize}
We say that $\msF^\heartsuit$ is a {\bf $\mr{GL}_n$-oper}  on $X^\mr{log}/S^\mr{log}$
if
it satisfies  the following two conditions:
\begin{itemize}
\item
For every $j=1, \cdots, n-1$,
$\nabla (\mcF^j)$ is contained in $\Omega \otimes \mcF^{j-1}$;
\item
For every $j=1, \cdots, n-1$,
the well-defined {\it $\mcO_X$-linear} morphism
\begin{align}
\mr{KS}_{\msF^\heartsuit}^j : \mcF^j/\mcF^{j+1} \rightarrow \Omega \otimes (\mcF^{j-1}/\mcF^j)
\end{align}
defined by $\overline{a} \mapsto \overline{\nabla (a)}$ for any local section $a \in \mcF^j$ (where $\overline{(-)}$'s denote the images in the respective quotients)  is an isomorphism.
\end{itemize}
The notion of an isomorphism between two $\mr{GL}_n$-opers  can be defined in a straightforward manner, and we omit the details.

For a flat vector bundle  $(\mcF, \nabla)$ on $X^\mr{log}/S^\mr{log}$ (i.e., $\mcF$ is a vector bundle on $X$ and $\nabla$ is an $S^\mr{log}$-connection on $\mcF$),
 the {\bf $p$-curvature} of  $\nabla$ is the $\mcO_X$-linear morphism ${^p}\psi_\nabla : \mcT^{\otimes p} \rightarrow \mcE nd_{\mcO_X} (\mcF)$  determined  by ${^p}\psi_\nabla (\partial^{\otimes p}) =  \nabla_\partial^p - \nabla_{\partial^{[p]}}$ for any local section $\partial \in \mcT$.
 Here,  
 $\nabla_{\partial'} := (\partial' \otimes \mr{id}_\mcF) \circ \nabla$ for $\partial' \in \mcT$, and $\partial^{[p]}$ denotes the $p$-th symbolic power of $\partial$ (i.e., ``$\partial \mapsto \partial^{(p)}$" asserted in 
 ~\cite[Proposition 1.2.1]{Ogu1}).
 A $\mr{GL}_n$-oper $\msF^\heartsuit := (\mcF, \nabla, \{ \mcF^j \}_j)$ is said to be {\bf dormant} if  the $p$-curvature  ${^p}\psi_\nabla$ of $\nabla$ vanishes  identically.

\subsection{$(\mr{GL}_n, \vartheta)$-, $(\mr{GO}_{2\ell +1}, \vartheta)$-, and $(\mr{GSp}_{2m}, \vartheta)$-opers}
 \label{p0140p2}

Recall  from ~\cite[Definition 4.31, (i)]{Wak5} that an {\bf $n$-theta characteristic} of $X^\mr{log}/S^\mr{log}$ is defined as a pair
$\vartheta := (\varTheta, \nabla_\vartheta)$
consisting of a line bundle $\varTheta$ on $X$ and an $S^\mr{log}$-connection  $\nabla_\vartheta$ on the line bundle $\mcT^{\otimes \frac{n (n-1)}{2}} \otimes \varTheta^{\otimes n}$.
An $n$-theta characteristic $\vartheta := (\varTheta, \nabla_\vartheta)$ is  said to be   {\bf dormant} if 
$\nabla_\vartheta$ has vanishing $p$-curvature.

We fix  a dormant $n$-theta characteristic $\vartheta := (\varTheta, \nabla_\vartheta)$ of $\msX$.
(According to the discussion in ~\cite[Section 4.6.4]{Wak5}, such an object always exists.)
We set
\begin{align}
\mcF_{\varTheta, n} := \mcD_{\leq n-1} \otimes \varTheta, \hspace{5mm} \mcF_{\varTheta, n}^j := \mcD_{\leq n  -j-1} \otimes \varTheta
 \ (j=0, \cdots, n).
\end{align}
Since $\mcF_{\varTheta, n}^{j}/\mcF_{\varTheta, n}^{j+1}$ ($j=0, \cdots, n-1$)
can be identified with $\mcT^{\otimes (n-j-1)} \otimes \varTheta$, we obtain a composite isomorphism
\begin{align} \label{EQ129}
\mr{det}(\mcF_{\varTheta, n}) \xrightarrow{\sim} \bigotimes_{j=0}^{n-1} \mcF_{\varTheta, n}^j/\mcF_{\varTheta, n}^{j+1} \xrightarrow{\sim}
\bigotimes_{j=0}^{n-1} (\mcT^{\otimes (n-j-1)} \otimes \varTheta)
\xrightarrow{\sim} \mcT^{\otimes \frac{n (n-1)}{2}} \otimes \varTheta^{\otimes n}.
\end{align}

\bde[cf. ~\cite{Wak5}, Definitions 4.36 and 5.4]  \label{pp15}
\begin{itemize}
\item[(i)]
By a {\bf $(\mr{GL}_n, \vartheta)$-oper} on $X^\mr{log}/S^\mr{log}$, we mean an $S^\mr{log}$-connection $\nabla^\diamondsuit$ on $\mcF_{\varTheta, n}$ such that
the collection 
\begin{align}
\nabla^{\diamondsuit \Rightarrow \heartsuit} := (\mcF_{\varTheta, n}, \nabla^\diamondsuit, \{ \mcF^j_{\varTheta, n} \}_{j=0}^n)
\end{align}
 forms a $\mr{GL}_n$-oper, and such that the $S^\mr{log}$-connection $\mr{det}(\nabla^\diamondsuit)$ on $\mr{det}(\mcF_{\varTheta, n})$ induced by  $\nabla^\diamondsuit$ coincides with $\nabla_\vartheta$ under the identification $\mr{det}(\mcT_{\varTheta, n}) = \mcT^{\otimes \frac{n(n-1)}{2}} \otimes \varTheta^{\otimes n}$ given by  \eqref{EQ129}.
\item[(ii)]
For  two $(\mr{GL}_n, \vartheta)$-opers $\nabla_1^\diamondsuit$ and  $\nabla_2^\diamondsuit$,  
an {\bf isomorphism of $(\mr{GL}_n, \vartheta)$-opers} from $\nabla_1^\diamondsuit$ to $\nabla_2^\diamondsuit$ is defined to be an isomorphism of $\mr{GL}_n$-opers from $\nabla_1^{\diamondsuit \Rightarrow \heartsuit}$ to $\nabla_2^{\diamondsuit \Rightarrow \heartsuit}$.
 \end{itemize}
  \ede

Next, suppose that $n < p$, and  write
\begin{align} \label{QW923}
\mcN_\bb := \mcT^{\otimes (n-1)} \otimes \varTheta^{\otimes 2}.
\index{${^\dagger}\mcN_\bb$}
\end{align}
According to  ~\cite[Proposition 4.22, (i)]{Wak5}, 
 there exists a unique $S^\mr{log}$-connection
$\nabla_{\mcN_\bb}$
on $\mcN_\bb$ satisfying   $\nabla_{\mcN_\bb}^{\otimes n} = \nabla_\bb^{\otimes 2}$ under the natural identification  $\mcN_\bb^{\otimes n} = (\mcT^{\otimes \frac{n(n-1)}{2}} \otimes \varTheta^{\otimes n})^{\otimes 2}$.
Here, given  a flat bundle $(\mcG, \nabla_\mcG)$ and $\ell \in \mbZ_{>0}$, we write $\nabla_\mcG^{\otimes \ell}$ for 
the $S^\mr{log}$-connection on the $\ell$-fold tensor product  $\mcG^{\otimes \ell}$ of $\mcG$ induced naturally from $\nabla_\mcG$.
The pair of $\mcN_\vartheta$ and $\nabla_{\mcN_\vartheta}$ defines  a flat  line bundle 
\begin{align} \label{QW1302}
\msN_\bb := (\mcN_\bb, \nabla_{\mcN_\bb}).
\end{align}

\bde \label{Def57}
Suppose further that $n =2\ell +1$ (resp., $n =2m$) for some positive integer $\ell$ (resp., $m$).
\begin{itemize}
\item[(i)]
By a {\bf $(\mr{GO}_{2\ell +1}, \vartheta)$-oper} (resp., a {\bf $(\mr{GSp}_{2m}, \vartheta)$-oper}) on $X^\mr{log}/S^\mr{log}$, we mean   a pair 
\begin{align} \label{QW611}
\nabla^\diamondsuit_{\sphericalangle} := (\nabla^\diamondsuit,  \omega)
\end{align}
consisting of a $(\mr{GL}_{n}, \bb)$-oper $\nabla^\diamondsuit$ on $X^\mr{log}/S^\mr{log}$ and a nondegenerate  symmetric (resp., skew-symmetric)
 $\mcO_X$-bilinear map  $\omega : \mcF_{\varTheta, n}^{\otimes 2} \rightarrow \mcN_\vartheta$ on $\mcF_{\varTheta, n}$ 
 satisfying the following two conditions:
 \begin{itemize}
 \item
 The $S^\mr{log}$-connection  $(\nabla^{\diamondsuit})^{\otimes 2}$ on $\mcF_{\varTheta, n}^{\otimes 2}$ induced by $\nabla^\diamondsuit$  is compatible with $\nabla_{\mcN_\vartheta}$ via 
 $\omega$;
 \item
 For any $j=0, \cdots, n$, 
 the equality $\mcF^{n-j} = (\mcF^{j})^\bot$
   holds, where 
   $ (\mcF^{j})^\bot$ denotes the $\mcO_X$-submodule of $\mcF$ consisting of local sections $v$ with $\omega (v, u) =0$ for all local sections  $u \in \mcF^j$.
 \end{itemize}
 \item[(ii)]
 For  two $(\mr{GO}_{2\ell +1}, \vartheta)$-opers (resp., $(\mr{GSp}_{2m}, \vartheta)$-opers) $\nabla_{\ang, 1}^\diamondsuit := (\nabla_1^\diamondsuit, \omega_1)$, $\nabla_{\ang, 2}^\diamondsuit := (\nabla_2^\diamondsuit, \omega_2)$,  
an {\bf isomorphism of $(\mr{GO}_{2\ell +1}, \vartheta)$-opers} (resp., {\bf of $(\mr{GSp}_{2m}, \vartheta)$-opers}) from $\nabla_1^\diamondsuit$ to $\nabla_2^\diamondsuit$ is defined to be an isomorphism of $\mr{GL}_n$-opers from $\nabla_1^{\diamondsuit \Rightarrow \heartsuit}$ to $\nabla_2^{\diamondsuit \Rightarrow \heartsuit}$ under  which the bilinear map $\omega_1$ corresponds to  $\omega_2$.
\end{itemize}
\ede

When $X^\mr{log}/S^\mr{log}$ arises from a pointed stable curve $\msX$ over the underlying scheme $S$ of $S^\mr{log}$, any $(G, \vartheta)$-oper on $X^\mr{log}/S^\mr{log}$ for $G \in \{ \mr{GL}_n, \mr{GO}_{2\ell +1}, \mr{GSp}_{2m} \}$  will be referred to as  a {\bf  $G$-oper  on $\msX$}.

\bde[cf. ~\cite{Wak5}, Definition 4.61] \label{Def41}
A $(\mr{GL}_n, \vartheta)$-oper (resp., $(\mr{GO}_{2\ell +1}, \vartheta)$-oper; $(\mr{GSp}_{2m}, \vartheta)$-oper) is said to be {\bf dormant}
if the underlying  $\mr{GL}_n$-oper is dormant.
\ede

According to ~\cite[Theorem 4.49]{Wak5},  changing the structure group by the quotient $\mr{GL}_n \twoheadrightarrow \mr{PGL}_n$ yields a bijection of sets 
\begin{align} \label{EQ33}
\begin{pmatrix} \text{the set of isomorphism classes of} \\  \text{$(\mr{GL}_{n}, \vartheta)$-opers on $X^\mr{log}/S^\mr{log}$} \end{pmatrix}
\xrightarrow{\sim}
\begin{pmatrix} \text{the set of isomorphism classes of} \\  \text{$\mfs \mfl_{n}$-opers on $X^\mr{log}/S^\mr{log}$}\end{pmatrix}.
\end{align}
For a $(\mr{GL}_n, \vartheta)$-oper $\nabla^\diamondsuit$, we shall write $\nabla^{\diamondsuit \Rightarrow \spadesuit}$ for the corresponding $\mfs \mfl_n$-oper.

It  follows from 
 ~\cite[Propositions 5.3, 5.6]{Wak5} that the assignment $\nabla^\diamondsuit \mapsto \nabla^{\diamondsuit \Rightarrow \spadesuit}$ extends to a bijection
\begin{align} \label{EQ3}
\begin{pmatrix} \text{the set of isomorphism classes of} \\  \text{$(\mr{GO}_{2\ell +1}, \vartheta)$-opers} \\\text{(resp., $(\mr{GSp}_{2m}, \vartheta)$-opers) on $X^\mr{log}/S^\mr{log}$}\end{pmatrix}
\xrightarrow{\sim}
\begin{pmatrix} \text{the set of isomorphism classes of} \\  \text{$\mfs \mfo_{2\ell +1}$-opers} \\\text{(resp., $\mfs \mfo_{2m}$-opers) on $X^\mr{log}/S^\mr{log}$}\end{pmatrix}.
\end{align}
For a $(\mr{GO}_{2\ell +1}, \vartheta)$-oper (resp., a $(\mr{GSp}_{2m}, \vartheta)$-oper) $\nabla^\diamondsuit_{\ang}$,
we denote by $\nabla_{\ang}^{\diamondsuit \Rightarrow \spadesuit}$ the corresponding $\mfs \mfo_{2\ell +1}$-oper (resp., $\mfs \mfp_{2m}$-oper).

Moreover, it restricts to a bijection
\begin{align} \label{EQ5}
\begin{pmatrix} \text{the set of isomorphism classes of} \\  \text{dormant $(\mr{GO}_{2\ell +1}, \vartheta)$-opers} \\\text{(resp., dormant $(\mr{GSp}_{2m}, \vartheta)$-opers)} \\ \text{on $X^\mr{log}/S^\mr{log}$}\end{pmatrix}
\xrightarrow{\sim}
\begin{pmatrix} \text{the set of isomorphism classes of} \\  \text{dormant $\mfs \mfo_{2\ell +1}$-opers} \\\text{(resp., dormant $\mfs \mfo_{2m}$-opers)} \\ \text{on $X^\mr{log}/S^\mr{log}$}\end{pmatrix}.
\end{align}
The formations of these bijections  commute with  base-changes to fs log schemes over $S^\mr{log}$.

\subsection{Example: Canonical dormant $\mfs \mfp_{p-1}$-oper} \label{SS99}

Suppose that $X$ is a geometrically connected,  proper, and  smooth curve  over $S$ (hence $X = X^\mr{log}$).
Denote by $\mcB$ the locally exact $1$-forms on $X$ relative to $S$, i.e., the image of the universal derivation $d : \mcO_X \rightarrow \Omega$.
This sheaf forms a rank $p-1$ vector bundle on the Frobenius twist $X^{(1)}$ of $X$ over $S$ via the underlying homeomorphism of the relative Frobenius morphism $F_{X/S} : X \rightarrow X^{(1)}$.
In particular, we obtain a short exact sequence of $\mcO_{X^{(1)}}$-modules
$0 \rightarrow \mcO_{X^{(1)}} \rightarrow F_{X/S*} (\mcO_{X}) \xrightarrow{} \mcB \rightarrow 0$.
Pulling-back  this sequence along $F_{X/S}$ yields  a short exact sequence of $\mcO_X$-modules
$0 \rightarrow \mcO_{X} \rightarrow \mcA \rightarrow \overline{\mcA} \rightarrow 0$, 
where $\mcA := F_{X/S}^* (F_{X/S*} (\mcO_X))$ and $\overline{\mcA} := F^*_{X/S} (\mcB)$.
Here, for a vector bundle $\mcG$ on $X^{(1)}$, we shall write $\nabla^\mr{can}_{\mcG}$ the canonical  $S \left(= S^\mr{log} \right)$-connection on $F_{X/S}^* (\mcG)$ with vanishing $p$-curvature, as constructed  in ~\cite[Theorem 5.1]{NKat0}.

Using the canonical connection $\nabla_{F_{X/S*}(\mcO_X)}^\mr{can}$, we construct a $p$-step filtration $\{ \mcA^j \}_{j=0}^p$ on the rank $p$ vector bundle $\mcA$ as follows.
\begin{align}
\mcA^0 &:= \mcA; \\
\mcA^1 &:= \mr{Ker} \left( \mcA \xrightarrow{q} \mcO_X\right); \\
\mcA^j &:= \mr{Ker} \left(\mcA^{j-1} \xrightarrow{\nabla_{F_{X/S*}(\mcO_X)}^\mr{can} |_{\mcA^{j-1}}} \Omega \otimes \mcA \xrightarrow{\mr{quotient}} 
\Omega \otimes \left( \mcA/\mcA^{j-1}\right) \right)
\end{align}
($j=2, \cdots, p$), where $q$ denotes the morphism $\mcA \rightarrow \mcO_X$ corresponding to the identity morphism of $F_{X/S*}(\mcO_X)$ via the adjunction relation ``$F^*_{X/S} (-) \dashv F_{X/S*}(-)$".
According to ~\cite[Section 5.3, Theorem]{JRXY}, ~\cite[Lemma 2.1]{Sun} (see also  ~\cite[Proposition 9.1]{Wak5}),
the collection of data 
\begin{align} \label{EQ101}
\msA^\heartsuit := (\mcA, \nabla_{F_{X/S*}(\mcO_X)}^\mr{can}, \{ \mcA^j \}_{j=0}^p)
\end{align}
 forms a dormant $\mr{GL}_p$-oper.
 Moreover,  the composite 
 \begin{align} \label{EQ1000}
 \mcO_X \hookrightarrow \mcA \twoheadrightarrow \mcA /\mcA^1 \left(= \mcA^0/\mcA^1 \right)
 \end{align}
  is an isomorphism.
  Here,   the first arrow denotes the pull-back along $F_{X/S}$ of the natural morphism $\mcO_{X^{(1)}} \rightarrow F_{X/S*}(\mcO_X)$ induced by $F_{X/S}$.

For each $j=1, 2,  \cdots, p-1$, we obtain a composite isomorphism
\begin{align} \label{EQ61}
\Omega^{\otimes j} \xrightarrow{\sim}\Omega^{\otimes j} \otimes (\mcA^0/\mcA^1) \xrightarrow{\sim} \cdots \xrightarrow{\sim} \Omega \otimes (\mcA^{j-1}/\mcA^{j})\xrightarrow{\sim}\mcA^{j}/\mcA^{j+1},
\end{align}
where the first arrow arises from \eqref{EQ1000}, and  for each $s  = 2, \cdots, j +1$ the  $s$-th arrow  arises from $\mr{KS}_{\msA^\heartsuit}^{s-1}$.
The  composite
\begin{align} \label{EQ55}
\mcF_{\Omega^{\otimes (p-1)}, p-1} \left(= \mcD_{\leq p-2} \otimes \Omega^{\otimes (p-1)} \right) & \xrightarrow{\mr{id} \otimes (\eqref{EQ61} \, \text{for} \, j=p-1)} \mcD_{\leq p-2} \otimes \mcA^{p-1}  \\
& \hookrightarrow  \mcD \otimes \mcA \notag \\
&  \rightarrow \mcA \notag \\
& \twoheadrightarrow \overline{\mcA} \notag
\end{align}
is verified to be an isomorphism, where the second arrow arises from the inclusions $\mcD_{\leq p-2} \hookrightarrow \mcD$ and 
$\mcA^{p-1} \hookrightarrow \mcA$, and the third arrow arises from $\nabla^\mr{can}_{F_{X/S*}(\mcO_X)}$.
We also obtain a composite isomorphism
\begin{align} \label{EQ62}
\mr{det}(\overline{\mcA}) \xrightarrow{\sim} \bigotimes_{j=1}^{p-1} \mcA^j/\mcA^{j+1} \xrightarrow{\sim} \bigotimes_{j=1}^{p-1}  \Omega^{\otimes j}
\xrightarrow{\sim} \mcT^{\otimes \frac{(p-1)(p-2)}{2}} \otimes (\Omega^{\otimes (p-1)})^{\otimes (p-1)}.
\end{align}

Denote by $\nabla_\mr{Ray}^\diamondsuit$ the $S$-connection on  $\mcF_{\Omega^{\otimes (p-1)}, p-1}$ corresponding to the $S$-connection $\nabla^\mr{can}_{\mcB}$ on $\overline{\mcA}$ via \eqref{EQ55}, and by $\nabla_{\vartheta_\mr{Ray}}$ the $S$-connection on $\mcT^{\otimes \frac{(p-1)(p-2)}{2}} \otimes (\Omega^{\otimes (p-1)})^{\otimes (p-1)}$ corresponding to the $S$-connection $\mr{det}(\nabla_\mcB^\mr{can})$ on  $\mr{det}(\overline{\mcA})$ via \eqref{EQ62}.
Then, the pair 
\begin{align} \label{EQ281}
\vartheta_\mr{Ray} := (\Omega^{\otimes (p-1)}, \nabla_{\vartheta_\mr{Ray}})
\end{align}
 defines a dormant $(p-1)$-theta characteristic, and $\nabla_\mr{Ray}^\diamondsuit$ forms a dormant $(\mr{GL}_{p-1}, \vartheta_\mr{Ray})$-oper.

Next, recall from  ~\cite[Section 4.1]{Ray}  that the assignment 
$(a, b) \mapsto C (a db)$ on $F_{X/S*} (\mcO_X) \times F_{X/S*} (\mcO_X)$ defines, upon passing through the quotient $F_{X/S*}(\mcO_X) \rightarrow \mcB$, a nondegenerate skew-symmetric  $\mcO_{X^{(1)}}$-bilinear map
$\mcB \otimes \mcB \rightarrow \Omega_{X^{(1)}/S}$
on $\mcB$.
Its  pull-back  along  $F_{X/S}$ determines an $\mcO_X$-bilinear map $F^*_{X/S} (\mcB) \otimes F^*_{X/S} (\mcB) \rightarrow \Omega_{X/S}^{\otimes p} \left(\cong F^*_{X/S}(\Omega_{X^{(1)}/S}) \right)$, with respect to which $\nabla_{\mcB}^{\otimes 2}$ is compatible with $\nabla^\mr{can}_{\Omega_{X^{(1)}/S}}$.
Using the identification $\mcF_{\Omega^{\otimes (p-1)}, p-1} = \overline{\mcA}$ given by \eqref{EQ55}, 
this defines a nondegenerate $\mcO_X$-linear map $\omega_\mr{Ray}$ on $\mcF_{\Omega^{\otimes (p-1)}, p-1}$.
The resulting pair
\begin{align} \label{EQ63}
\nabla_{\ang, \mr{Ray}}^\diamondsuit := (\Omega^{\otimes (p-1)}, \omega_{\mr{Ray}})
\end{align}
specifies a dormant $(\mr{GSp}_{p-1}, \vartheta_\mr{Ray})$-oper, and hence, it induces a dormant $\mfs \mfp_{p-1}$-oper $\nabla_{\ang, \mr{Ray}}^{\diamondsuit \Rightarrow \spadesuit}$.

\subsection{Twisting with a flat line bundle} \label{SS511}

Let us retain  the notation before Section \ref{SS99}, and suppose that we are given a flat line bundle 
 $\msL := (\mcL, \nabla_\mcL)$ on $X^\mr{log}/S^\mr{log}$, i.e., 
 $\mcL$ is a line bundle on $X$ and $\nabla_\mcL$ denotes an $S^\mr{log}$-connection on $\mcL$.
We then define 
 \begin{align}
 \vartheta \otimes \mcL := (\varTheta \otimes \mcL, \nabla_\vartheta \otimes \nabla_\mcL^{\otimes n}),
 \end{align}
 where 
  $ \nabla_\vartheta \otimes \nabla_\mcL^{\otimes n}$ denotes the $S^\mr{log}$-connection on $\mcT^{\otimes \frac{n (n-1)}{2}} \otimes (\varTheta \otimes \mcL)^{\otimes n} \left(\cong (\mcT^{\otimes \frac{n (n-1)}{2}} \otimes \varTheta^{\otimes n}) \otimes \mcL^{\otimes n} \right)$  induced from $\nabla_\vartheta$ and $\nabla_\mcL^{\otimes n}$.
 This pair 
  forms 
a (dormant) $n$-theta characteristic of $X^\mr{log}/S^\mr{log}$ (cf. ~\cite[Section 4.6.5]{Wak5}).

Now, let $\nabla^\diamondsuit$ be a  $(\mr{GL}_n, \vartheta)$-oper on $X^\mr{log}/S^\mr{log}$.
The tensor product $\mcL\otimes \mcD$ carries  a  left $\mcD$-action 
\begin{align}
\mcD \otimes (\mcL \otimes \mcD) \rightarrow \mcL \otimes \mcD
\end{align}
extending its $\mcO_X$-module structure and
arising   from both the $\mcD$-action corresponding to $\nabla_\mcL$ and the natural left $\mcD$-action on $\mcD$ itself.
This morphism restricts to an isomorphism of $\mcO_X$-modules
\begin{align}
\label{Eq502}
\mcD_{\leq n-1} \otimes \mcL \left(= \mcD_{\leq n-1} \otimes (\mcL \otimes \mcD_{\leq 0}) \right)
\xrightarrow{\sim} \mcL \otimes \mcD_{\leq n-1}.
\end{align}
Tensoring  this isomorphism with the identity morphism of $\varTheta$ yields  an isomorphism
\begin{align} \label{EQ556}
\mcF_{\mcL \otimes \varTheta, n} \left(= \mcD_{\leq n-1} \otimes \mcL \otimes \varTheta \right) \xrightarrow{\sim} \left(\mcL \otimes \mcD_{\leq n-1} \otimes \varTheta = \right) \mcL \otimes \mcF_{\varTheta, n}
\end{align}
Via  this isomorphism, $\nabla^\diamondsuit \otimes \nabla_\mcL$ corresponds to an $S^\mr{log}$-connection 
\begin{align}
\nabla^{\diamondsuit}_{\otimes \msL} : \mcF_{\mcL \otimes \varTheta, n} \rightarrow \Omega \otimes \mcF_{\mcL \otimes \varTheta, n}
\end{align}
on $\mcF_{\mcL \otimes \varTheta, n}$, which forms
 a  $(\mr{GL}_n, \vartheta \otimes \msL)$-oper.

When $\nabla^\diamondsuit$ is dormant and $\nabla_\mcL$ has vanishing $p$-curvature,  the resulting $(\mr{GL}_n, \vartheta \otimes \msL)$-oper $\nabla^{\diamondsuit}_{\otimes \msL}$ is dormant.
By construction, the associated $\mfs \mfl_n$-opers 
$\nabla^{\diamondsuit \Rightarrow \spadesuit}$ and $(\nabla^\diamondsuit_{\otimes \msL})^{\Rightarrow \spadesuit}$ are isomorphic.
Moreover, if we   write $\msL^\vee := (\mcL^\vee, \nabla_\mcL^\vee)$, where $\nabla_\mcL^\vee$ denotes the $S^\mr{log}$-connection on the dual $\mcL^\vee$ of $\mcL$ induced  from $\nabla_\mcL$,
then the $(\mr{GL}_n, \vartheta)$-oper $\nabla^\diamondsuit$ is isomorphic to $(\nabla^\diamondsuit_{\otimes \msL})_{\otimes \msL^\vee}$ under the natural  identification $(\vartheta \otimes \msL)\otimes \msL^\vee = \vartheta$.

Next, suppose that  $p > n = 2\ell +1$ (resp., $p> n =2m$) for some positive integer $\ell$ (resp., $m$).
Let $\nabla_\ang^\diamondsuit := (\nabla^\diamondsuit, \omega)$ be a $(\mr{GO}_{2\ell +1}, \vartheta)$-oper (resp., a $(\mr{GSp}_{2m}, \vartheta)$-oper) on $X^\mr{log}/S^\mr{log}$.
Then,  
\begin{align} \label{QW602}
\omega_{\otimes \mcL} := \omega \otimes (\mr{id}_{\mcL^{\otimes 2}}) : (\mcF_\varTheta \otimes \mcL)^{\otimes 2} \left(= \mcF_\varTheta^{\otimes 2} \otimes \mcL^{\otimes 2} \right) \rightarrow \mcN_\vartheta \otimes \mcL^{\otimes 2}
\end{align}
defines a symmetric (resp., a skew-symmetric) $\mcO_X$-bilinear map on $\mcF_\varTheta \otimes \mcL$, and  
 the resulting pair
 \begin{align} \label{EQ16}
 \nabla_{\ang, \otimes \msL}^\diamondsuit := (\nabla^\diamondsuit_{\otimes \msL}, \omega_{\otimes \mcL})
 \end{align}
 forms a $(\mr{GO}_{2\ell +1}, \vartheta \otimes \msL)$-oper (resp., a $(\mr{GSp}_{2m}, \vartheta \otimes \msL)$-oper).
It is clear that  $(\nabla^\diamondsuit_{\ang})^{\Rightarrow \spadesuit} \cong (\nabla^\diamondsuit_{\ang, \otimes \msL})^{\Rightarrow \spadesuit}$.

\subsection{Self-dual  $(\mr{GL}_{n}, \vartheta)$-opers} \label{SS300}

Let $\msF^\heartsuit := (\mcF, \nabla, \{ \mcF^j \}_{j=0}^n)$ be a $\mr{GL}_n$-oper on $X^\mr{log}/S^\mr{log}$.
For each $j=0, \cdots, n$, we regard $\mcF^{\vee j} := (\mcF/\mcF^{n-j})^\vee$ as a subbundle of  the dual $\mcF^\vee$ of $\mcF$.
The resulting collection
\begin{align}
\msF^{\heartsuit \BB} := (\mcF^\vee, \nabla^\vee, \{ \mcF^{\vee j}\}_{j=0}^n),
\end{align} 
where $\nabla^\vee$ denotes the $S^\mr{log}$-connection on $\mcF^\vee$ induced naturally from $\nabla$,  forms a $\mr{GL}_n$-oper.
If $\msF^\heartsuit$ is dormant,  then $\msF^{\heartsuit \BB}$ is dormant.

Let $\vartheta$ be as above, and write $\varTheta^\BB := \Omega^{\otimes (n-1)} \otimes \varTheta^\vee$.
Then, we have a composite isomorphism
\begin{align}
\mcT^{\otimes \frac{n (n-1)}{2}} \otimes  (\varTheta^\BB)^{\otimes n} 
&\xrightarrow{\sim}
\mcT^{\otimes \frac{n (n-1)}{2}} \otimes \mcT^{\otimes (-n(n-1))} \otimes (\varTheta^\vee)^{\otimes n} \\
&\xrightarrow{\sim}
(\mcT^{\otimes \frac{n(n-1)}{2}} \otimes \varTheta^{\otimes n})^\vee. \notag
\end{align}
Via this composite, the dual $\nabla_\vartheta^\vee$ of $\nabla_\vartheta$ can be regarded as an $S^\mr{log}$-connection on $\mcT^{\otimes \frac{n(n-1)}{2}} \otimes (\varTheta^\BB)^{\otimes n}$.
Thus, we obtain a (dormant) $n$-theta characteristic
\begin{align}
\vartheta^\BB := (\varTheta^\BB, \nabla_\vartheta^\vee)
\end{align}
of $X^\mr{log}/S^\mr{log}$.
We shall call $\vartheta^\BB$ the {\bf dual} of $\vartheta$ (cf. ~\cite[Section 5.3.2]{Wak5}).
It is immediate that $(\vartheta^\BB)^\BB$ can be identified with $\vartheta$.
Also, under the natural identification $\varTheta^\BB \otimes \mcN_\vartheta = \varTheta$, the following equality holds:
\begin{align} \label{EQ9}
\vartheta^\BB \otimes \msN_\vartheta =  \vartheta.
\end{align}

Now, let $\nabla^\diamondsuit$ be a $(\mr{GL}_n, \vartheta)$-oper on $X^\mr{log}/S^\mr{log}$.
Since 
$\mcF_{\varTheta, n}^{\vee n-1} = (\mcF_{\varTheta, n}/\mcF_{\varTheta, n}^1)^\vee = (\mcT^{\otimes (n-1)} \otimes \varTheta)^\vee = \varTheta^\BB$,
we obtain  an inclusion $\varTheta^\BB \hookrightarrow \mcF_{\varTheta, n}^\vee$.
The composite
\begin{align}
\mcF_{\varTheta^\BB, n} \left(= \mcD_{\leq n-1} \otimes \varTheta^\BB \right) \hookrightarrow \mcD \otimes \mcF_{\varTheta, n}^\vee \xrightarrow{\nabla^{\diamondsuit\vee}} \mcF_{\varTheta, n}^\vee
\end{align}
is an isomorphism.
The dual $\nabla^{\diamondsuit \vee}$ of $\nabla^\diamondsuit$ corresponds to an $S^\mr{log}$-connection 
\begin{align}
\nabla^{\diamondsuit \BB} : \mcF_{\varTheta^\BB, n} \rightarrow \Omega \otimes \mcF_{\varTheta^\BB, n}
\end{align}
on $\mcF_{\varTheta^\BB, n}$ via this isomorphism, and  it  forms 
a $(\mr{GL}_n, \vartheta^\BB)$-oper.
Since   $\nabla^{\diamondsuit \BB \BB} = \nabla^\diamondsuit$,  the resulting assignment $\nabla^\diamondsuit \mapsto \nabla^{\diamondsuit \BB}$ defines a bijection
\begin{align} \label{EQ600}
\begin{pmatrix} \text{the set of isomorphism classes of} \\ \text{$(\mr{GL}_n, \vartheta)$-opers on $X^\mr{log}/S^\mr{log}$} \end{pmatrix}
\xrightarrow{\sim}
\begin{pmatrix} \text{the set of isomorphism classes of} \\ \text{$(\mr{GL}_n, \vartheta^\BB)$-opers on $X^\mr{log}/S^\mr{log}$} \end{pmatrix}.
\end{align}

\bde \label{Def33}
A $(\mr{GL}_n, \vartheta)$-oper  $\nabla^\diamondsuit$ is said to be {\bf self-dual}
if it is isomorphic to 
the $(\mr{GL}_n, \vartheta)$-oper 
$(\nabla^{\diamondsuit \BB})_{\otimes \msN_\vartheta}$ associated with $\nabla^\diamondsuit$  under the identification
\eqref{EQ9}.
\ede

Under the assumption  that $p > n = 2\ell +1$ (resp., $p>n =2m$) for some positive integer $\ell$ (resp., $m$),  let us take  a $(\mr{GO}_{2\ell +1}, \vartheta)$-oper (resp., a $(\mr{GSp}_{2m}, \vartheta)$-oper) $\nabla_\ang^\diamondsuit := (\nabla^\diamondsuit, \omega)$  on $X^\mr{log}/S^\mr{log}$.
Then, the isomorphism $\mcF_{\varTheta, n} \xrightarrow{\sim} \mcF_{\varTheta, n}^\vee \otimes \mcN_{\vartheta}$ induced by  the $\mcO_X$-bilinear map  $\omega$ defines an isomorphism  of $(\mr{GL}_n, \vartheta)$-opers $\nabla^\diamondsuit \xrightarrow{\sim} (\nabla^{\diamondsuit \BB})_{\otimes \msN_\vartheta}$.
This means that  the underlying $(\mr{GL}_n, \vartheta)$-oper $\nabla^\diamondsuit$,  denoted occasionally  by $\nabla^\diamondsuit_{\ang \Rightarrow \emptyset}$,  is self-dual.

Conversely, let $\nabla^\diamondsuit$ be a self-dual $(\mr{GL}_n, \vartheta)$-oper on $X^\mr{log}/S^\mr{log}$.
We fix an isomorphism of $(\mr{GL}_n, \vartheta)$-opers $\eta : \nabla^\diamondsuit \xrightarrow{\sim} (\nabla^{\diamondsuit \BB})_{\otimes \msN_\vartheta}$, whose underlying morphism of vector bundles is of the form $\mcF_{\varTheta, n} \xrightarrow{\sim} \mcF_{\varTheta, n}^\vee \otimes \mcN_\vartheta$.
This isomorphism induces a nondegenerate $\mcO_X$-bilinear map
\begin{align}
\omega_{\nabla^\diamondsuit} : \mcF_{\varTheta, n}^{\otimes 2} \rightarrow \mcN_\vartheta
\end{align}
on $\mcF_{\varTheta, n}$, with respect to which  the $S^\mr{log}$-connections $\nabla^{\diamondsuit \otimes 2}$ is compatible with $\nabla_{\mcN_\vartheta}$.
Note that $\omega_{\nabla^\diamondsuit}$ depends on the choice of $\eta$, but by ~\cite[Proposition 4.25]{Wak5} it is uniquely determined up to composition with an automorphism of the flat line bundle $\msN_\vartheta$.

\bpr \label{Prop88}
Suppose that $p > n = 2\ell +1$ (resp., $p > n = 2m$) for some positive integer $\ell$ (resp., $m$).
Then, 
the pair 
\begin{align} \label{Eq228}
\nabla^\diamondsuit_{\emptyset \Rightarrow \ang} := (\nabla^\diamondsuit, \omega_{\nabla^\diamondsuit})
\end{align}
forms a $(\mr{GO}_{2\ell +1}, \vartheta)$-oper (resp., a $(\mr{GSp}_{2m}, \vartheta)$-oper).
\epr
\begin{proof}
Denote by $\omega_{\nabla^\diamondsuit, 1}$ and  $\omega_{\nabla^\diamondsuit, 2}$ the isomorphisms of $\mr{GL}_n$-opers $\nabla^{\diamondsuit \Rightarrow \heartsuit} \xrightarrow{\sim} ((\nabla^{\diamondsuit \BB})_{\otimes \msN_\vartheta})^{\Rightarrow \heartsuit}$ given by assigning $v \mapsto \omega_{\nabla^\diamondsuit} (v \otimes ( -))$ and  $v \mapsto \omega_{\nabla^\diamondsuit} ((-) \otimes  v)$, respectively,  for any local section $v \in \mcF_{\varTheta, n}$.
By  ~\cite[Proposition 4.25]{Wak5}, 
there exists a global section  $\mu$  of $\mcO_X$ satisfying the equality $\omega_{\nabla^\diamondsuit, 2} = \mu \cdot \omega_{\nabla^\diamondsuit, 1}$.
Under the assumption that  $p > n = 2\ell +1$ (resp., $p > n = 2m$), the problem is to show the equality  $\mu =1$ (resp., $\mu =-1$).
To this end,  we are free to replace $X$ with its scheme-theoretic dense open subscheme.
Hence, we assume, without loss of generality,  that 
there exists generators $\partial \in H^0 (X, \mcT)$,  $v \in H^0 (X, \varTheta)$  over $X$ of $\mcT$ and  $\varTheta$, respectively.

First,  we consider the non-resp'd assertion, i.e., the orthogonal case.
The section $v' := (\nabla^\diamondsuit)^\ell (v) \in \mcF_{\varTheta, 2\ell +1}^\ell$ is mapped to a generator of $\mcF_{\varTheta, 2\ell +1}^\ell/\mcF_{\varTheta, 2\ell +1}^{\ell +1}$ via the natural quotient, and we have $\omega_{\nabla^\diamondsuit} (v' \otimes v') \in H^0 (X, \mcO_X^\times)$. 
Observe  that 
\begin{align}
 \omega_{\nabla^\diamondsuit} (v' \otimes v') =(\omega_{\nabla^\diamondsuit, 2} (v')) (v') =  (\mu \cdot \omega_{\nabla^\diamondsuit, 1} (v')) v' =\mu \cdot \omega_{\nabla^\diamondsuit} (v' \otimes v').
\end{align}
Since  $\omega_{\nabla^\diamondsuit} (v' \otimes v')$ is invertible, the resulting  equality  $\omega_{\nabla^\diamondsuit} (v' \otimes v') = \mu \cdot \omega_{\nabla^\diamondsuit} (v' \otimes v')$ forces $\mu =1$,
as desired.

Next, we prove the resp'd assertion, i.e., the symplectic case. 
Write $v'' := (\nabla^\diamondsuit)^m (v) \in \mcF_{\varTheta, 2m}^m$.
Then,  the sections $v'' \in\mcF_{\varTheta, 2m}^m$, $\nabla^\diamondsuit (v'')  \in \mcF_{\varTheta, 2m}^{m-1}$ are
 mapped, via the natural quotients,  to generators  of  $ \mcF_{\varTheta, 2m}^{m}/ \mcF_{\varTheta, 2m}^{m+1}$  and $ \mcF_{\varTheta, 2m}^{m-1}/ \mcF_{\varTheta, 2m}^{m}$, respectively.
It follows that $\omega_{\nabla^\diamondsuit} (\nabla^{\diamondsuit}(v'') \otimes v'') \in H^0 (X, \mcO_X^\times)$.
Since $(\mcF_{\varTheta, 2m}^m)^\perp = \mcF_{\varTheta, 2m}^m$,  the following chain of equalities holds:
\begin{align}
\omega_\nabla (\nabla^\diamondsuit (v'') \otimes v'') + \omega_{\nabla} (v'' \otimes \nabla^\diamondsuit (v'')) 
& = \omega_{\nabla^\diamondsuit} ((\nabla^{\diamondsuit})^{\otimes 2}(v'' \otimes v''))  \\
& = \nabla_{\mcN_\vartheta}(\omega_\nabla (v'' \otimes v'')) \notag \\
& = \nabla_{\mcN_\vartheta} (0) \notag \\
&= 0. \notag
\end{align}
Hence,  we have 
\begin{align}
 \omega_\nabla (\nabla^\diamondsuit (v'') \otimes v'') & =(\omega_{\nabla, 2} (v'')) (\nabla^\diamondsuit (v'')) \\
 & =  (\mu \cdot \omega_{\nabla, 1} (v'')) (\nabla^\diamondsuit(v'')) \notag \\
 & =\mu \cdot \omega_\nabla (v'' \otimes \nabla^\diamondsuit ( v'')) \notag \\
 &=  - \mu \cdot \omega_\nabla (\nabla^\diamondsuit (v'') \otimes v''). \notag
\end{align}
Since  $ \omega_\nabla (\nabla^\diamondsuit (v'') \otimes v'')$ is invertible, the resulting equality $\omega_\nabla (\nabla^\diamondsuit (v'') \otimes v'') = - \mu \cdot \omega_\nabla (\nabla^\diamondsuit (v'') \otimes v'')$ forces 
  $\mu = -1$.
  This  completes  the proof of the assertion.
\end{proof}

By the above proposition,  the following statement follows  immediately; it refines  ~\cite[Proposition 5.10]{Wak5}.

\bco \label{Prop89}
Suppose that $p >n = 2\ell +1$ (resp., $p> n = 2m$) for some positive integer $\ell$ (resp., $m$).
Then, the assignments $\nabla^\diamondsuit_{\ang} \mapsto  \nabla^\diamondsuit_{\ang \Rightarrow \emptyset}$ and  $\nabla^\diamondsuit \mapsto \nabla^\diamondsuit_{\emptyset \Rightarrow \ang}$ (cf. \eqref{Eq228})  together  define a bijection of sets
\begin{align} \label{EQ1}
\begin{pmatrix} \text{the set of isomorphism classes of} \\ \text{$(\mr{GO}_{2 \ell +1}, \vartheta)$-opers} \\ \text{(resp., $(\mr{GSp}_{2 m}, \vartheta)$-opers) on $X^\mr{log}/S^\mr{log}$} \end{pmatrix}
\xrightarrow{\sim}
\begin{pmatrix} \text{the of isomorphism classes of} \\ \text{self-dual $(\mr{GL}_{n}, \vartheta)$-opers} \\ \text{ on $X^\mr{log}/S^\mr{log}$}\end{pmatrix}.
\end{align}
Moreover, this bijection  restricts to a bijection of sets
\begin{align} \label{EQ2}
\begin{pmatrix} \text{the set of isomorphism classes of} \\ \text{dormant  $(\mr{GO}_{2 \ell +1}, \vartheta)$-opers} \\ \text{(resp., dormant $(\mr{GSp}_{2 m}, \vartheta)$-opers)} \\ \text{on $X^\mr{log}/S^\mr{log}$} \end{pmatrix}
\xrightarrow{\sim}
\begin{pmatrix} \text{the set of isomorphism classes of } \\ \text{self-dual dormant  $(\mr{GL}_{n}, \vartheta)$-opers} \\ \text{on $X^\mr{log}/S^\mr{log}$} \end{pmatrix}.
\end{align}
Finally, the formations of these bijections commute with  base-changes to fs log schemes over $S^\mr{log}$.
\eco

\subsection{Data of radii}
   \label{SS600}

In the rest of this section, we 
assume that $X^\mr{log}/S^\mr{log}$ arises from an $r$-pointed stable curve  $\msX := (f : X \rightarrow S, \{ \sigma_i \}_{i=1}^r)$ of genus $g$ over the underlying scheme $S$ of $S^\mr{log}$.
Unless otherwise stated,  we further assume  that $r > 0$.

Denote by $\widetilde{\Xi}_{n}$ the collection of all subsets of $\mbF_p$ with cardinality $n$.
We say that two elements $\{ e_{1,1}, \cdots, e_{1, n} \}$, $\{ e_{2,1}, \cdots, e_{2, n} \}$ of $\widetilde{\Xi}_n$ are {\bf equivalent}
if there exists an element $c \in \mbF_p$ satisfying $\{ e_{1,1}, \cdots, e_{1, n} \} = \{ e_{2,1} +c, \cdots, e_{2, n} +c \}$.
This binary relation in  $\widetilde{\Xi}_{n}$ defines an equivalence relation.
Thus,  we obtain the quotient set
\begin{align} \label{EQ10}
\Xi_n
\end{align}
of  $\widetilde{\Xi}_{n}$  by this equivalence relation, as  well as  
the natural projection  $\pi : \widetilde{\Xi}_n \twoheadrightarrow \Xi_n$.

Next, let  $(-)^\BB$ denote the involution on $\widetilde{\Xi}_n$ 
given by assigning $\{ e_{1}, \cdots, e_n\}^\BB :=  \{-e_1, \cdots, -e_n\}$ (for $e_1, \cdots, e_n \in \mbF_p$).
This induces an involution on $\Xi_n$, which will be written by the same notation $(-)^\BB$.
Write 
\begin{align} \label{EQ29}
\widetilde{\Xi}_{n, \mr{sym} } \ \left(\text{resp.,} \  \Xi_{n, \mr{sym}}  \right)
\end{align}
for  the subset of $\widetilde{\Xi}_n$ (resp.,  $\Xi_n$) consisting of elements 
invariant under $(-)^\BB$. 

For an element $e := \{ e_1, \cdots, e_n \}$ of $\widetilde{\Xi}_n$, we shall write $e^\BBB := \mbF_p \setminus e \in \widetilde{\Xi}_{p-n}$.
The assignments $e \mapsto e^\BBB$ and $e \mapsto e^\BBBB := e^{\BBB \BB}$, respectively,  determine  bijections $\widetilde{\Xi}_n \xrightarrow{\sim} \widetilde{\Xi}_{p-n}$, and they induce  well-defined bijections
\begin{align} \label{EQ304}
(-)^\BBB :  \Xi_n \xrightarrow{\sim} \Xi_{p-n} \ \ \ \text{and} \ \ \ (-)^\BBBB :  \Xi_n \xrightarrow{\sim} \Xi_{p-n},
\end{align}
respectively, satisfying 
 $(-)^{\BBB\BBB} = (-)^{\BBBB \BBBB} = \mr{id}_{\Xi_n}$.
Moreover, if $e$ belongs to $\widetilde{\Xi}_{n, \mr{sym}}$, then the element $e^\BBB$ belongs to $\widetilde{\Xi}_{p-n, \mr{sym}}$ and coincides with $e^\BBBB$.
The bijections in   \eqref{EQ304} restricts to a common bijection 
 \begin{align} \label{EQ28e}
 (-)^\BBBB : \Xi_{n, \mr{sym}} \xrightarrow{\sim} \Xi_{p-n, \mr{sym}}.
 \end{align}

\subsection{Radii  of dormant opers} \label{SS6002}

Let $\vartheta := (\varTheta, \nabla_\vartheta)$ be a dormant $n$-theta characteristic  of $X^\mr{log}/S^\mr{log}$ and
$\nabla^\diamondsuit$ a dormant $(\mr{GL}_n, \vartheta)$-oper.
Choose  an element $i$ of $\{1, \cdots, n \}$.
The {\it residue matrix} (or, the {\it monodromy operator}, in the sense of ~\cite[Definition 4.42]{Wak5}) of $\nabla^\diamondsuit$ at the $i$-th marked point $\sigma_i$ is the element of $\mr{End}_{\mcO_S} (\sigma_i^*(\mcF_{\varTheta, n}))$ defined to be the composite
\begin{align}
\mu_i (\nabla^\diamondsuit) : \sigma_i^* (\mcF_{\varTheta, n}) & \xrightarrow{\sigma_i(\nabla^\diamondsuit)} \sigma_i^*(\Omega \otimes \mcF_{\varTheta, n}) \\
& \xrightarrow{\sim} \sigma_i^*(\Omega) \otimes \sigma_i^* (\mcF_{\varTheta, n}) \notag \\
& \xrightarrow{\mr{Res}_i \otimes\mr{id}}
\left( \mcO_S \otimes \sigma_i^*(\mcF_{\varTheta, n}) = \right) \sigma_i^*(\mcF_{\varTheta, n}),
\end{align}
where $\mr{Res}_i$ denotes the usual residue isomorphism $\sigma_i^*(\Omega) \xrightarrow{\sim} \mcO_S$.
It follows from ~\cite[Proposition 8.4, (ii)]{Wak5}
that there exists a unique element 
\begin{align} \label{EQ11}
 e_i (\nabla^\diamondsuit) := \{ e_{i, 1}, \cdots, e_{i, n} \}
 \end{align}
of $\widetilde{\Xi}_n$
 defining the eigenvalues of $\mu_i(\nabla^\diamondsuit)$ (counted with multiplicity).
This element $e_i (\nabla^\diamondsuit)$ is called the {\bf exponent} of $\nabla^\diamondsuit$ at $\sigma_i$.
By the definition of $\nabla^{\diamondsuit \BB}$, the equality  $e_i (\nabla^{\diamondsuit \BB}) = e_i (\nabla^\diamondsuit)^\BB$ holds.

For an $r$-tuple $e := (e_i)_{i=1}^r \in \widetilde{\Xi}_n^r\left(:= \widetilde{\Xi}_n \times \cdots \times \widetilde{\Xi}_n \right)$, we say that $\nabla^\diamondsuit$ is {\bf of exponents $e$} if $e_i = e_i (\nabla^\diamondsuit)$ for every $i=1, \cdots, r$. When $r =0$, any dormant $(\mr{GL}_n, \vartheta)$-oper is said to be {\bf of exponents $\emptyset$}.

Given a  dormant $\mfs \mfl_n$-oper $\msE^\spadesuit$ on $\msX$, we associate to each marked point $\sigma_i$
  an element 
\begin{align}
\rho_i (\msE^\spadesuit) := \pi (e_{i} (\nabla^\diamondsuit)) \in \Xi_n
\end{align}
defined to be the image, via $\pi$, of the exponent $e_i (\nabla^\diamondsuit)$  at $\sigma_i$ of the corresponding dormant $(\mr{GL}_n, \vartheta)$-oper $\nabla^\diamondsuit$.
This element  does not depend on the choice of $\vartheta$;
in other words, it  depends only on the isomorphism class of $\msE^\spadesuit$.
We refer  to  $\rho_i (\msE^\spadesuit)$ as the {\bf radius} of $\msE^\spadesuit$ at $\sigma_i$.

For $\rho := (\rho_i)_{i=1}^r \in \Xi_n^r$,
we say that $\msE^\spadesuit$ is {\bf of radii $\rho$} if $\rho_i = \rho_i ({\msE^\spadesuit})$ for every $i =1, \cdots, r$.
When $r = 0$, any dormant $\mfs \mfl_n$-oper is said to be {\bf of radii $\emptyset$}.
The bijection \eqref{EQ600} induces, via \eqref{EQ33},  a bijection
\begin{align} \label{EQ601}
\begin{pmatrix} \text{the set of isomorphism classes} \\ \text{of dormant $\mfs \mfl_n$-opers} \\ \text{of radii $\rho$ on $\msX$}\end{pmatrix}
\xrightarrow{\sim}
\begin{pmatrix} \text{the set of isomorphism classes} \\ \text{of dormant $\mfs \mfl_n$-opers} \\ \text{of radii $\rho^\BB$ on $\msX$}\end{pmatrix},
\end{align}
where $\rho^\BB := (\rho_i^\BB)_{i=1}^r$.

Now, suppose that $p > n = 2\ell +1$ (resp., $p > n=2m$) for some positive integer $\ell$ (resp., $m$), and let
  $\nabla_{\ang}^\diamondsuit := (\nabla^\diamondsuit, \omega)$ be a dormant $(\mr{GO}_{2\ell +1}, \vartheta)$-oper (resp., a dormant $(\mr{GSp}_{2m}, \vartheta)$-oper).
For each $i=1, \cdots, r$,
we obtain an element $e_i (\nabla^\diamondsuit) \in \widetilde{\Xi}_n$.
Recall that taking the dual of a logarithmic connection transforms its residue matrix into $(-1)$ times its transpose.
 Thus,  the self-duality for $\nabla^\diamondsuit$  implies that $e_i (\nabla^\diamondsuit)$ lies in $\widetilde{\Xi}_{n, \mr{sym}}$.
 
 Also,  let $\msE^\spadesuit$ be a dormant $\mfs \mfo_{2\ell +1}$-oper  (resp.,  a dormant $\mfs \mfp_{2m}$-oper) on $\msX$.
 It corresponds to a dormant $(\mr{GO}_{2\ell +1}, \vartheta)$-oper (resp., a dormant $(\mr{GSp}_{2m}, \vartheta)$-oper) $\nabla_{\ang}^\diamondsuit$ via  \eqref{EQ5}.
 Then, the resulting element 
\begin{align} \label{EQ13}
\rho_i (\msE^\spadesuit) :=  \pi (e_i (\nabla^\diamondsuit))
\end{align}
does not depend on the choice of $\vartheta$ and lies in $\Xi_{n, \mr{sym}}$.
We refer to $\rho_i (\msE^\spadesuit)$ as the {\bf radii} of $\msE^\spadesuit$ at $\sigma_i$.
For $\rho := (\rho_i )_{i=1}^r \in \Xi_{n, \mr{sym}}^{r}$,
we  say that a dormant $\mfs \mfo_{2\ell +1}$-oper (resp., a dormant $\mfs \mfp_{2m}$-oper) $\msE^\spadesuit$ is {\bf of radii $\rho$} if $\rho_i = \rho_i ({\msE^\spadesuit})$ for every $i =1, \cdots, r$.
When $r = 0$, any dormant $\mfs \mfo_{2\ell +1}$-oper (resp., dormant $\mfs \mfp_{2m}$-oper) is said to be {\bf of radii $\emptyset$}.

\vspace{10mm}
\section{Duality between  dormant $\mfs \mfo_{2\ell +1}$- and $\mfs \mfp_{2m}$-opers} \label{S20}

In this section, we establish a duality between dormant $\mfs \mfo_{2\ell +1}$- and $\mfs \mfp_{2m}$-opers under the condition that $p-1 =2(\ell + m) +1$ (cf. Theorems \ref{Prop461} and \ref{QW1199}).
This duality is obtained by restricting the  classical duality  between dormant  $\mfs \mfl_n$- and $\mfs \mfl_{p-n}$-opers proved in ~\cite{Wak2}.
As an application, we achieve a detailed understanding of the geometric structure of the moduli spaces parametrizing  dormant $\mfs \mfo_{2\ell +1}$- and $\mfs \mfp_{2m}$-opers (cf. Corollary \ref{Cor21}).

\subsection{Canonical dormant $\mfs \mfl_{p}$-oper} \label{SS046}

As before,
let $S^\mr{log}$ be an fs log scheme over $k$ and $f^\mr{log} : X^\mr{log} \rightarrow S^\mr{log}$ a log curve over $S^\mr{log}$.
We continue to use the notation introduced in the previous section for 
$X^\mr{log}/S^\mr{log}$.

Let $\mcL$ be a line bundle on $X$.
Observe that 
 \begin{align} \label{EQ14}
 \mcT^{\otimes \frac{p (p-1)}{2}} \otimes \mcL^{\otimes p} \cong (\mcT^{\otimes \frac{p -1}{2}}\otimes \mcL)^{\otimes p}
\cong F_{X/S}^* (\mcW), 
 \end{align}
 where  $\mcW$ denotes the vector bundle on $X^{(1)}$  defined as the base-change of $\mcT^{\otimes \frac{p -1}{2}}\otimes \mcL$ via the absolute Frobenius morphism $S \rightarrow S$ of $S$.
Just as in  the non-logarithmic case recalled in Section \ref{SS99},
$F_{X/S}^* (\mcW)$ carries  a canonical $S^\mr{log}$-connection with vanishing $p$-curvature;
 the corresponding $S^\mr{log}$-connection on  $\mcT^{\otimes \frac{p (p-1)}{2}} \otimes \mcL^{\otimes p}$ via \eqref{EQ14}
 is denoted by $\nabla_{\vartheta_\mcL}$. 
Thus, we obtain a dormant $p$-theta characteristic
\begin{align} \label{EQ100}
\vartheta_\mcL := (\mcL, \nabla_{\vartheta_\mcL})
\end{align}
of $X^\mr{log}/S^\mr{log}$.

Let ${^p}\psi$ denote the $\mcO_X$-linear morphism $\mcT^{\otimes p} \rightarrow \mcD$ determined  by $\partial^{\otimes p} \mapsto \partial^p - \partial^{[p]}$ for any local section $\partial \in \mcT$.
Note that  the composite
\begin{align} \label{Eq222}
\mcF_{\mcL, p} \left(= \mcD_{\leq p -1} \otimes \mcL \right) \xrightarrow{\mr{inclusion}} \mcD \otimes \mcL \twoheadrightarrow (\mcD \otimes \mcL) / \langle \mr{Im}({^p}\psi \otimes \mr{id}_\mcL)\rangle
\end{align}
is an isomorphism,
where 
$(\mcD \otimes \mcL) / \langle \mr{Im}({^p}\psi \otimes \mr{id}_\mcL)\rangle$ denotes 
the quotient  of the left $\mcD$-module $\mcD \otimes \mcL$ by the $\mcD$-submodule generated by the image of ${^p}\psi \otimes \mr{id}_\mcL : \mcT_\X^{\otimes p} \otimes \mcL \rightarrow \mcD \otimes \mcL$.
 The natural  $\mcD$-module structure on $(\mcD \otimes \mcL) / \langle \mr{Im}({^p}\psi \otimes \mr{id}_\mcL)\rangle$
corresponds to  an $S^\mr{log}$-connection $\nabla^\diamondsuit_{\mcL}$ on $\mcF_{\mcL, p}$  via this composite isomorphism.
 The collection of data
 \begin{align} \label{Eq300}
 \msF^\heartsuit_\mcL := (\mcF_{\mcL, p}, \nabla^\diamondsuit_{\mcL}, \{ \mcF^j_{\mcL, p} \}_{j=0}^p)
 \end{align}
  forms a dormant 
  $\mr{GL}_{p}$-oper on $X^\mr{log}/S^\mr{log}$, and moreover $\nabla_{\mcL}^\diamondsuit$ defines a dormant $(\mr{GL}_{p}, \vartheta_\mcL)$-oper.
  (If $X/S$ is a geometrically connected,  proper and smooth curve, then it follows from ~\cite[Proposition 5.1.1]{Wak2} that the dormant $\mr{GL}_p$-oper $\mcA^\heartsuit$ introduced in \eqref{EQ101} is isomorphic to $\msF^\heartsuit_{\Omega^{\otimes (p-1)}}$.)
 In particular, it corresponds, via \eqref{EQ33}, to a dormant $\mfs \mfl_{p}$-oper 
 \begin{align}
 \msE^\spadesuit_\mcL
 \end{align}
 on $X^\mr{log}/S^\mr{log}$.
  Since $\Xi_p$ is a singleton consisting of the element 
 \begin{align} \label{EQ301}
 \rho_\mr{full} := \pi (\mbF_p),
 \end{align}
the radius  of $\msE^\spadesuit_\mcL$ at every marked point coincides with $\rho_\mr{full}$.

\subsection{Duality between  dormant $\mfs \mfl_n$- and $\mfs \mfl_{p-n}$-opers I} \label{SS11}

Let $n$ be an integer with $1 < n < p$.
In this subsection, we recall from ~\cite{Wak2} the duality between  dormant $\mfs \mfl_n$- and $\mfs \mfl_{p-n}$-opers.

Let us fix a dormant $n$-theta characteristic $\vartheta := (\varTheta, \nabla_\vartheta)$ of $X^\mr{log}/S^\mr{log}$.
Write $\varTheta^\BBB := \mcT^{\otimes n} \otimes \varTheta \left(= (\varTheta^\BB)^\vee \right)$.
Then,  we obtain  a sequence of natural isomorphisms
\begin{align}
\mcT^{\otimes \frac{(p-n)(p-n-1)}{2}} \otimes (\varTheta^\BBB)^{\otimes (p-n)}
&\xrightarrow{\sim}
\mcT^{\otimes \frac{(p-n)(p+n-1)}{2}} \otimes \varTheta^{\otimes (p-n)} \\
&\xrightarrow{\sim} (\mcT^{\otimes \frac{p (p-1)}{2}} \otimes \varTheta^{\otimes p}) \otimes (\mcT^{\otimes \frac{n(n-1)}{2}} \otimes \varTheta^{\otimes n})^\vee. \notag
\end{align}
Then, $\nabla_{\vartheta_\mcL} \otimes \nabla_\vartheta^\vee$ corresponds to  an $S^\mr{log}$-connection $\nabla_{\vartheta^\BBB}$ on $\mcT^{\otimes \frac{(p-n)(p-n-1)}{2}} \otimes (\varTheta^\BBB)^{\otimes (p-n)}$.
It follows that the pair
\begin{align}
\vartheta^\BBB := (\varTheta^\BBB, \nabla_{\vartheta^\BBB})
\end{align}
forms a dormant $(p-n)$-theta characteristic.

On the other hand, 
the canonical $S^\mr{log}$-connection $\nabla_{\mcT^{\otimes p}}$  on $\mcT^{\otimes p}$ defined as in the discussion of Section \ref{SS046} yields   a flat line bundle $\msT :=(\mcT^{\otimes p}, \nabla_{\mcT^{\otimes p}})$.
 Then, we have
\begin{align} \label{EQ105}
(\vartheta^\BBB)^\BBB = \vartheta \otimes \msT
\end{align}
under the natural identification $\left( (\varTheta^\BBB)^\BBB = \right) \mcT^{\otimes (p-n)} \otimes (\mcT^{\otimes n} \otimes \varTheta) = \mcT^{\otimes p} \otimes \varTheta$.

Next, let $\nabla^\diamondsuit$ be a dormant $(\mr{GL}_n, \vartheta)$-oper on $X^\mr{log}/S^\mr{log}$.
The inclusion $\varTheta \hookrightarrow \mcF_{\varTheta, n}$ extends uniquely  to a $\mcD$-linear morphism $\mcD \otimes \varTheta \rightarrow \mcF_{\varTheta, n}$, where $\mcF_{\varTheta, n}$ is equipped with  the $\mcD$-action corresponding to $\nabla^\diamondsuit$.
Since $\nabla^\diamondsuit$ has vanishing $p$-curvature, 
this morphism factors through the quotient $\mcD \otimes \varTheta \twoheadrightarrow \mcD \otimes \varTheta / \langle \mr{Im} ({^p}\psi \otimes \mr{id}_\varTheta)\rangle$.
Therefore, under the identification $\mcF_{\varTheta, p} = \mcD \otimes \varTheta / \langle \mr{Im} ({^p}\psi \otimes \mr{id}_\varTheta)\rangle$ given by \eqref{Eq222} for $\mcL = \varTheta$, we obtain a morphism of flat bundles
\begin{align}
\nu_{\nabla^\diamondsuit} : (\mcF_{\varTheta, p}, \nabla_{\varTheta}^\diamondsuit) \rightarrow (\mcF_{\varTheta, n}, \nabla^\diamondsuit).
\end{align}

The composite
$\mcF_{\varTheta, p}^{p-n} \hookrightarrow \mcF_{\varTheta, p} \xrightarrow{\nu_{\nabla^\diamondsuit}} \mcF_{\varTheta, n}$ is an isomorphism.
This implies the surjectivity of 
$\nu_{\nabla^\diamondsuit}$, so  the composite 
\begin{align} \label{Eq223}
\mr{Ker}(\nu_{\nabla^\diamondsuit}) \hookrightarrow \mcF_{\varTheta, p} \twoheadrightarrow \mcF_{\varTheta, p} / \mcF_{\varTheta, p}^{p-n}
\end{align}
is an isomorphism.
The line subbundle $\mcF_{\varTheta, p}^{p-n-1}/\mcF_{\varTheta, p}^{p-n}$ of  $\mcF_{\varTheta, p} / \mcF_{\varTheta, p}^{p-n}$ 
can be naturally  identified with $\varTheta^\BBB$.
Hence, by passing through  \eqref{Eq223}, 
we regard $\varTheta^\BBB$  as a line subbundle of $\mr{Ker}(\nu_{\nabla^\diamondsuit})$.
The inclusion $\varTheta^\BBB \hookrightarrow \mr{Ker}(\nu_{\nabla^\diamondsuit})$ extends to a  morphism of $\mcD$-modules $\mcD \otimes \varTheta^\BBB \rightarrow \mr{Ker}(\nu_{\nabla^\diamondsuit})$, where   $\mr{Ker}(\nu_{\nabla^\diamondsuit})$ is equipped with the $\mcD$-action  obtained by restricting  the $\mcD$-action on $\mcF_{\varTheta, p}$ corresponding to  $\nabla_{\varTheta}^\diamondsuit$.
Then, the composite
\begin{align} \label{Eq460}
\mcF_{\varTheta^\BBB, p-n}  \xrightarrow{\mr{inclusion}} \mcD \otimes \varTheta^\BBB \xrightarrow{} \mr{Ker}(\nu_{\nabla^\diamondsuit})
\end{align}
is verified to be  an isomorphism.

Denote by 
\begin{align} \label{EQ209}
\nabla^{\diamondsuit \BBB}
\end{align}
 the $S^\mr{log}$-connection on $\mcF_{\varTheta^\BBB, p-n}$ corresponding to that on $ \mr{Ker}(\nu_{\nabla^\diamondsuit})$ via \eqref{Eq460}.
In particular, we obtain  the following short exact sequence of flat bundles
\begin{align} \label{EQ29}
0 \rightarrow (\mcF_{\varTheta^\BBB, p-n}, \nabla^{\diamondsuit \BBB}) \rightarrow (\mcF_{\varTheta}, \nabla_{\varTheta}^\diamondsuit)
\rightarrow (\mcF_{\varTheta, n}, \nabla^\diamondsuit)
\rightarrow 0.
\end{align}
By construction,  $\nabla^{\diamondsuit \BBB}$ forms a dormant $(\mr{GL}_{p-n}, \vartheta^\BBB)$-oper on $X^\mr{log}/S^\mr{log}$.
Moreover, by applying the same construction 
 to
  $\nabla^{\diamondsuit \BBB}$ itself, we obtain a dormant  $(\mr{GL}_n, \vartheta \otimes \msT)$-oper $(\nabla^{\diamondsuit \BBB})^\BBB$ under the identification $(\vartheta^\BBB)^\BBB = \vartheta \otimes \msT$ in  \eqref{EQ105}.

\subsection{Duality between  dormant $\mfs \mfl_n$- and $\mfs \mfl_{p-n}$-opers II} \label{SS164}

In what follows, we describe  the relationship between $\nabla^\diamondsuit$ and $(\nabla^{\diamondsuit \BBB})^\BBB$.
The tensor product $\mcT^{\otimes p} \otimes \mcD$ admits a structure of left $\mcD$-module
\begin{align}
\mcD \otimes (\mcT^{\otimes p} \otimes \mcD) \rightarrow \mcT^{\otimes p} \otimes \mcD
\end{align}
arising from both $\nabla_{\mcT^{\otimes p}}$ and the natural left $\mcD$-action on $\mcD$ itself.
This morphism restricts to an isomorphism
\begin{align}
\label{Eq502}
\mcD_{\leq n-1} \otimes \mcT^{\otimes p} \left(= \mcD_{\leq n-1} \otimes (\mcT^{\otimes p} \otimes \mcD_{\leq 0}) \right)
\xrightarrow{\sim} \mcT^{\otimes p} \otimes \mcD_{\leq n-1}.
\end{align}
Tensoring  this isomorphism with the identity morphism of $\varTheta$  yields  an isomorphism
\begin{align} \label{EQ556}
\mcF_{\mcT^{\otimes p} \otimes \varTheta, n} \left(= \mcD_{\leq n-1} \otimes \mcT^{\otimes p} \otimes \varTheta \right) \xrightarrow{\sim} \left(\mcT^{\otimes p} \otimes \mcD_{\leq n-1} \otimes \varTheta = \right) \mcT^{\otimes p} \otimes \mcF_{\varTheta, n}.
\end{align}

\bpr \label{Prop456}
Under the natural identification $\mcF_{\mcT^{\otimes p} \otimes \varTheta, n} = \mcT^{\otimes p} \otimes \mcF_{\varTheta, n}$ given by  \eqref{EQ556}, 
the $(\mr{GL}_n, \vartheta \otimes \msT)$-oper $(\nabla^{\diamondsuit \BBB})^\BBB$ is isomorphic to   $\nabla^\diamondsuit_{\otimes \msT}$.
In particular,  the $\mfs \mfl_n$-oper $((\nabla^{\diamondsuit \BBB})^{\BBB})^{\Rightarrow \spadesuit}$ is isomorphic to  $\nabla^{\diamondsuit \Rightarrow \spadesuit}$.
\epr
\begin{proof}
The assertion follows from the various definitions involved.
\end{proof}

Consequently, the assignment $\nabla^\diamondsuit \mapsto \nabla^{\diamondsuit \BBB}$ defines a bijection of sets
\begin{align} \label{EQ110}
\begin{pmatrix}\text{the set of isomorphism classes  of} \\ \text{dormant $(\mr{GL}_n, \vartheta)$-opers} \\ \text{on $X^\mr{log}/S^\mr{log}$}\end{pmatrix}
\cong 
\begin{pmatrix}\text{the set of isomorphism classes of} \\ \text{dormant $(\mr{GL}_{p-n}, \vartheta^\BBB)$-opers} \\ \text{on $X^\mr{log}/S^\mr{log}$}\end{pmatrix},
\end{align}
which was already established in ~\cite[Theorem 4.3.1]{Wak2}.
(Strictly speaking,
this bijection coincides with the duality constructed in 
that reference,
except for the difference in the operation of applying $(-)^\BB$.)
Moreover,   it induces, via  \eqref{EQ33}, a bijection of sets
\begin{align} \label{EQ120}
\begin{pmatrix}\text{the  set of isomorphism classes  of}
 \\ \text{dormant $\mfs \mfl_n$-opers on $X^\mr{log}/S^\mr{log}$}
 \end{pmatrix}
\cong 
\begin{pmatrix}\text{the set of isomorphism classes of}
 \\ \text{dormant $\mfs \mfl_{p-n}$-opers on $X^\mr{log}/S^\mr{log}$}
 \end{pmatrix}.
\end{align}

Suppose further that 
$X^\mr{log}/S^\mr{log}$ 
arises from an $r$-pointed stable curve $\msX := (f : X \rightarrow S, \{ \sigma_i \}_{i=1}^r)$ of genus $g$ over the underlying scheme $S$ of $S^\mr{log}$.
Let $\nabla^\diamondsuit$  be a dormant $(\mr{GL}_n, \vartheta)$-oper.
Since the $(\mr{GL}_p, \vartheta_\mcL)$-oper $\nabla_\mcL^\diamondsuit$ is of exponents $(\mbF_p, \cdots, \mbF_p)$,
the exactness of the sequence \eqref{EQ29} implies that 
the exponent $e_i (\nabla^{\diamondsuit \BBB})$ coincides with $e_i (\nabla^{\diamondsuit})^\BBB$ for every $i =1, \cdots, r$.
Hence, for $e := (e_i)_{i=1}^r \in \widetilde{\Xi}_{n}^r$, the bijection \eqref{EQ110} restricts to a bijection
\begin{align} \label{EQ116}
\begin{pmatrix}\text{the set of isomorphism classes  of} \\ \text{dormant $(\mr{GL}_n, \vartheta)$-opers} \\ \text{of exponents $e$ on $\msX$}\end{pmatrix}
\cong 
\begin{pmatrix}\text{the set of isomorphism classes of} \\ \text{dormant $(\mr{GL}_{p-n}, \vartheta^\BBB)$-opers} \\ \text{of exponents $e^\BBB$ on $\msX$}\end{pmatrix},
\end{align}
where $e^\BBB := (e^\BBB_1, \cdots, e^\BBB_r)$.
Moreover, for $\rho := (\rho_i)_{i=1}^r \in \Xi_{n}^r$,
the bijection \eqref{EQ120} restricts to a bijection
\begin{align} \label{EQ126}
\begin{pmatrix}\text{the  set of isomorphism classes  of}
 \\ \text{dormant $\mfs \mfl_n$-opers} \\ \text{of radii $\rho$ on $\msX$}
 \end{pmatrix}
\cong 
\begin{pmatrix}\text{the set of isomorphism classes of}
 \\ \text{dormant $\mfs \mfl_{p-n}$-opers} \\ \text{of radii $\rho^\BBB$ on $\msX$}
 \end{pmatrix},
\end{align}
where $\rho^\BBB := (\rho_1^\BBB, \cdots, \rho_r^\BBB)$.

\subsection{Projective connections on a log curve I} \label{SS0f46}

Let $n$ be an integer with $1 \leq n < p$.
For each $j \in \mbZ_{\geq 0}$, we set
\begin{align}
{^\dagger}\mcD_{\leq j, \varTheta} := \mcH om_{\mcO_X} (\varTheta^\vee, (\Omega^{\otimes n} \otimes \varTheta^\vee) \otimes \mcD_{\leq  j}),
\  \ \ {^\ddagger}\mcD_{\leq j, \varTheta} := \mcH om_{\mcO_X} (\mcT^{\otimes n}\otimes \varTheta, \mcD_{\leq j} \otimes \varTheta).
\end{align}
As explained in ~\cite[Remark 4.12]{Wak5},
each  section of ${^\dagger}\mcD_{\leq j, \varTheta}$  can be regarded, under the assumption that $j <p$,  a linear differential operator from $\varTheta^\vee$ to $\Omega^{\otimes n} \otimes \varTheta^\vee$ in the usual sense.
Note  that  there is  a composite of natural isomorphisms
\begin{align}
\beta_j :   {^\dagger}\mcD_{\leq j, \varTheta} \xrightarrow{\sim} \Omega^{\otimes n} \otimes \varTheta^\vee \otimes \mcD_{\leq j} \otimes \varTheta \xrightarrow{\sim} {^\ddagger}\mcD_{\leq j, \varTheta}.
\end{align}
Also, since   $\mcD_{\leq  j}/\mcD_{\leq j-1}$ is identified with $\mcT^{\otimes j}$,
we obtain a composite
\begin{align}
\Sigma : {^\dagger}\mcD_{\leq j, \varTheta} \twoheadrightarrow {^\dagger}\mcD_{\leq j, \varTheta}/{^\dagger}\mcD_{\leq j-1, \varTheta}
\xrightarrow{\sim} \Omega^{\otimes (n-j)}.
\end{align}
If $D$ is a local section of  ${^\dagger}\mcD_{\leq j, \varTheta}$, then its image $\Sigma (D)$ under $\Sigma$ is called the {\bf principal symbol} of $D$.

Let $D^\clubsuit : \varTheta^\vee \rightarrow \Omega^{\otimes n} \otimes \varTheta^\vee \otimes \mcD_{\leq n}$ be
a global section of  ${^\dagger}\mcD_{\leq n, \varTheta}$ with $\Sigma (D^\clubsuit)  =1 \in H^0 (X, \mcO_X)$.
It induces  a left $\mcD$-module $\mcD\otimes \varTheta / \langle \mr{Im} (\beta_n (D^\clubsuit)) \rangle$, i.e., the quotient of $\mcD \otimes \varTheta$ by the $\mcD$-submodule generated by the image of the morphism 
\begin{align}
\beta_n (D^\clubsuit) : \mcT^{\otimes n} \otimes \varTheta \rightarrow \mcD_{\leq n} \otimes \varTheta \left(\subseteq \mcD \otimes \varTheta \right).
\end{align}
Since $\Sigma (D^\clubsuit) =1$ and $n < p$, the composite
\begin{align} \label{EQ20}
\mcF_{\varTheta, n} \xrightarrow{\mr{inclusion}} \mcD\otimes \varTheta \xrightarrow{\mr{quotient}} \mcD\otimes \varTheta /  \langle \mr{Im} (\beta_n (D^\clubsuit)) \rangle
\end{align}
is an isomorphism of $\mcO_X$-modules.
The $\mcD$-action on $ \mcD \otimes \varTheta /  \langle \mr{Im} (\beta_n (D^\clubsuit)) \rangle$ determines, via this composite isomorphism, 
an $S^\mr{log}$-connection
\begin{align}
D^{\clubsuit \Rightarrow \diamondsuit} : \mcF_{\varTheta, n} \rightarrow \Omega \otimes \mcF_{\varTheta, n}
\end{align}
on $\mcF_{\varTheta, n}$.
One may verify that the collection of data
\begin{align}
D^{\clubsuit \Rightarrow \heartsuit} := (\mcF_{\varTheta, n}, D^{\clubsuit \Rightarrow \diamondsuit}, \{ \mcF_{\varTheta, n}^j \}_{j=0}^n)
\end{align}
forms a $\mr{GL}_n$-oper on $X^\mr{log}/S^\mr{log}$.

\bde[cf. ~\cite{Wak5}, Definition 4.37, (ii)]
We say that $D^\clubsuit$ is an {\bf $(n, \vartheta)$-projective connection} on $X^\mr{log}/S^\mr{log}$ if
the associated $S^\mr{log}$-connection $D^{\clubsuit \Rightarrow \diamondsuit}$ forms a $(\mr{GL}_n, \vartheta)$-oper.
\ede

According to ~\cite[Theorem 4.41]{Wak5}, the assignment $D^\clubsuit \mapsto D^{\clubsuit \Rightarrow \diamondsuit}$ defines a bijection of sets
\begin{align} \label{EQ131}
\begin{pmatrix} \text{the set of $(n, \vartheta)$-projective} \\ \text{connections on $X^\mr{log}/S^\mr{log}$} \end{pmatrix}
\xrightarrow{\sim}
\begin{pmatrix} \text{the set  isomorphism classes of} \\ \text{$(\mr{GL}_n, \vartheta)$-opers on $X^\mr{log}/S^\mr{log}$} \end{pmatrix}.
\end{align}
The formation of this bijection  commutes with  base-changes over fs log schemes over $S^\mr{log}$, as well as  restrictions to open subschemes of $X$.

\subsection{Projective connections on a log curve II} \label{SS231}

In  this subsection,
we suppose that 
$f^\mr{log} : X^\mr{log} \rightarrow S^\mr{log}$ is strict, in the sense of ~\cite[S\,1.2]{ILL}.
Moreover, suppose that there exists a global section
$x \in H^0 (X, \mcO_X)$ such that the dual $\partial$ of $dx$ generates $\mcT$, i.e., $\mcT = \mcO_X \cdot  \partial$.
The sheaf of crystalline  differential operators admits a decomposition $\mcD = \bigoplus_{i \in \mbZ_{\geq 0}} \mcO_X \cdot \partial^i$.
Under the identification  $\mcT^{\otimes \frac{n (n-1)}{2}} \otimes \mcO_X^{\otimes n} \left(=\mcO_X \cdot \partial^{\otimes \frac{n(n-1)}{2}}\right) = \mcO_X$ given by the correspondence $ \partial^{\otimes \frac{n(n-1)}{2}} \leftrightarrow 1$, the universal derivation on $\mcO_X$ determines an $S^\mr{log}$-connection $\nabla_o$ on  $\mcT^{\otimes \frac{n (n-1)}{2}} \otimes \mcO_X^{\otimes n}$.
In particular, we obtain, for any $n$,  a dormant $n$-theta characteristic $o_n := (\mcO_X, \nabla_o)$ of $X^\mr{log}/S^\mr{log}$.
When
the integer $n$ is clear from the context, we shall simply write
  $o :=o_n$.

Now, let $D^\clubsuit$ be an $(n, o)$-projection connection on $X^\mr{log}/S^\mr{log}$ such that the corresponding $(\mr{GL}_n, o)$-oper $D^{\clubsuit \Rightarrow \diamondsuit}$ is dormant.
Under the identification $\Omega^{\otimes n} = \mcO_X$ via $\partial^{n}$,
the operator $D^\clubsuit$ can be regarded as a global section of $\mcD_{\leq n} := \bigoplus_{i=0}^{n} \mcO_X \cdot \partial^i \left(= {^\dagger}\mcD_{\leq n, \mcO_X} \cong {^\ddagger}\mcD_{\leq n, \mcO_X} \right)$.
Then, {\it the equality $\mr{det}(\mcF_{\mcO_X, n}) = \nabla_o$ is equivalent to the condition that the coefficient of $\partial^{n -1}$ in $D^\clubsuit$ vanishes} (cf. ~\cite[Remark 4.30]{Wak5}).
Recall that  $(\mcF_{\mcO_X, n}, D^{\clubsuit \Rightarrow \diamondsuit})$ is isomorphic to $\mcD/ \langle D^\clubsuit \rangle$ (equipped with the  $S^\mr{log}$-connection corresponding to the natural $\mcD$-action on $\mcD$ itself) via \eqref{EQ20}.
Since $D^{\clubsuit \Rightarrow \diamondsuit}$ has vanishing $p$-curvature, 
the element $\partial^p \left(= \partial^p - \partial^{[p]} \right)$ vanishes when considered as a global section  of $\mcD/ \langle D^\clubsuit \rangle$.
Hence,  one can find $D_\BBB^\clubsuit \in \mcD_{\leq p-n}$ satisfying the equality  $D_\BBB^\clubsuit \cdot D^\clubsuit = \partial^p$.
This equality implies that the coefficient of $\partial^{p-n-1}$ in $D_\BBB^\clubsuit$ vanishes.
By the italicized statement described above (applied with  $n$ replaced by $p-n$),
we see that $D_\BBB^\clubsuit$ forms a $(p-n, o)$-projective connection on $X^\mr{log}/S^\mr{log}$.

\ble \label{Lem54}
Let us keep the above notation and assumptions.
Then, the following assertions hold.
\begin{itemize}
\item[(i)]
The dormant $(\mr{GL}_{p-n}, o)$-oper $(D^{\clubsuit \Rightarrow \diamondsuit})^{\BBB}$ is isomorphic to $(D_\BBB^\clubsuit)^{\Rightarrow \diamondsuit}$.
\item[(ii)]
The equality $D^\clubsuit \cdot D_\BBB^\clubsuit = \partial^p$ holds.
\end{itemize}
\ele
\begin{proof}
First, we consider assertion (i).
The sequence \eqref{EQ29} applied to  $D^{\clubsuit \Rightarrow \diamondsuit}$ yields a short exact sequence of flat bundles
\begin{align} \label{EQ133}
0 \rightarrow (\mcF_{\mcO_X, p-n}, (D^{\clubsuit \Rightarrow \diamondsuit})^{\BBB}) \rightarrow (\mcF_{\mcO_X, p}, \nabla^\diamondsuit_{\mcO_X})
\rightarrow (\mcF_{\mcO_X, n}, D^{\clubsuit \Rightarrow \diamondsuit})\rightarrow 0.
\end{align}
On the other hand, the equality $D_\BBB^\clubsuit \cdot D^\clubsuit = \partial^p$
  yields 
 the following short exact sequence of flat bundles
\begin{align} \label{EQ25}
0 \rightarrow \mcD/\langle D_\BBB^\clubsuit\rangle \rightarrow \mcD/\langle \partial^p \rangle  \rightarrow \mcD/\langle D^\clubsuit \rangle \rightarrow 0,
\end{align}
where the second arrow arises from the inclusion $\langle D^\clubsuit \rangle / \langle \partial^p \rangle \hookrightarrow \mcD / \langle \partial^p \rangle$ under the natural identification $\langle D^\clubsuit \rangle / \langle \partial^p \rangle \cong \mcD/\langle D_\BBB^\clubsuit \rangle$, and the third arrow denotes the natural projection.
Observe that the  square diagram
\begin{align} \label{EQ180}
\vcenter{\xymatrix@C=46pt@R=36pt{
 (\mcF_{\mcO_X}, \nabla^\diamondsuit_{\mcO_X}) \ar[r] \ar[d]_-{\wr} & (\mcF_{\mcO_X, n}, D^{\clubsuit \Rightarrow \diamondsuit})\ar[d]^-{\wr} \\
 \mcD/ \langle \partial^p \rangle \ar[r] & \mcD/\langle D^\clubsuit \rangle
 }}
\end{align}
is commutative, where the upper and lower horizontal arrows are the third arrows in \eqref{EQ133} and \eqref{EQ25}, respectively, and the vertical arrows are given by \eqref{EQ20}.
By  \eqref{EQ133} and \eqref{EQ25}, the kernels of  the upper and lower horizontal arrows yield an isomorphism of flat bundles 
$(\mcF_{\mcO_X, p-n}, (D^{\clubsuit \Rightarrow \diamondsuit})^{\BBB}) \xrightarrow{\sim} \mcD/\langle D_\BBB^\clubsuit \rangle$.
The bijectivity  \eqref{EQ131} for $\vartheta = o$ implies 
 $(D^{\clubsuit \Rightarrow \diamondsuit})^{\BBB} \cong (D_\BBB^\clubsuit)^{\Rightarrow \diamondsuit}$, which proves assertion (i).

Next, we consider assertion (ii).
Applying the discussion preceding this lemma to $D_\BBB^\clubsuit$, we obtain an element $D_{\BBB\BBB}^\clubsuit \in H^0 (X, \mcD_{\leq n})$ with $D_{\BBB \BBB}^\clubsuit \cdot D_\BBB^\clubsuit = \partial^p$.
Observe the following chain of equalities:
\begin{align}
D_{\BBB \BBB}^\clubsuit \cdot \partial^p = D_{\BBB \BBB}^\clubsuit \cdot  (D^\clubsuit_{\BBB} \cdot D^\clubsuit) = (D_{\BBB \BBB}^\clubsuit \cdot D^\clubsuit_{\BBB}) \cdot D^\clubsuit = \partial^p \cdot D^\clubsuit  = D^\clubsuit  \cdot \partial^p,
\end{align}
where the last  equality follows from the well-known fact that $\partial^p$ lies in the center of $\mcD$.
Hence, we have $(D_{\BBB \BBB}^\clubsuit - D^\clubsuit) \cdot \partial^p =0$, which implies $D_{\BBB \BBB}^\clubsuit  = D^\clubsuit$.
It follows that $D^\clubsuit \cdot D_\BBB^\clubsuit = D_{\BBB \BBB}^\clubsuit \cdot D_\BBB^\clubsuit  = \partial^p$, thereby completing the proof of assertion (ii).
\end{proof}

The above lemma will be used in the proof of the following assertion.

\bpr \label{Lemma3}
Let us keep the above notation, and suppose that the dormant  $(\mr{GL}_{n}, o)$-oper  $D^{\clubsuit \Rightarrow \diamondsuit}$ is self-dual.
Then,  the dormant $(\mr{GL}_{p-n}, o)$-oper $(D^{\clubsuit \Rightarrow \diamondsuit})^\BBB$ is self-dual.
\epr
\begin{proof}
Just as in \eqref{EQ25}, the equality $D_\BBB^\clubsuit \cdot D^\clubsuit = \partial^p$ resulting from Lemma \ref{Lem54}, (ii),
yields a short exact sequence of flat bundles
\begin{align}
0 \rightarrow \mcD / \langle D^\clubsuit \rangle \rightarrow \mcD/\langle \partial^p \rangle \rightarrow \mcD/\langle D_\BBB^\clubsuit \rangle \rightarrow 0.
\end{align}
By passing through  the isomorphisms \eqref{EQ20} for $D^\clubsuit$, $D_\BBB^\clubsuit$, and $\partial^p$,
this sequence is transformed into  a short exact sequence 
\begin{align} \label{EQ200}
0 \rightarrow  (\mcF_{\mcO_X, n}, D^{\clubsuit \Rightarrow \diamondsuit})
\rightarrow (\mcF_{\mcO_X, p}, \nabla^\diamondsuit_{\mcO_X})
\rightarrow (\mcF_{\mcO_X, p-n}, (D^{\clubsuit \Rightarrow \diamondsuit})^\BBB)
\rightarrow 0.
\end{align}

Since  $\nabla_{\mcO_X}^\diamondsuit$  is verified to be self-dual (cf.  ~\cite[Proposition 3.7.2]{Wak2}),
we obtain a square diagram
\begin{align} \label{EQ180}
\vcenter{\xymatrix@C=46pt@R=36pt{
 (\mcF_{\mcO_X, p}, \nabla^\diamondsuit_{\mcO_X}) \ar[r] \ar[d]_-{\wr} & (\mcF_{\mcO_X, n}, D^{\clubsuit \Rightarrow \diamondsuit})\ar[d]^-{\wr} \\
  (\mcF_{\mcO_X, p}^\vee, \nabla^{\diamondsuit \vee}_{\mcO_X}) \ar[r] & 
  (\mcF_{\mcO_X, n}^\vee, (D^{\clubsuit \Rightarrow \diamondsuit})^\vee),
 }}
\end{align}
where the upper horizontal arrow denotes the third arrow in \eqref{EQ133}, the lower horizontal arrow denotes the dual of the second arrow in  \eqref{EQ200}, and the right-hand vertical arrow arises from the assumption that $D^{\clubsuit \Rightarrow \diamondsuit}$.
By ~\cite[Proposition 3.1.3]{Wak2} and ~\cite[Proposition 4.25]{Wak5}, this diagram becomes commutative after possibly replacing the right-hand (or the left-hand) vertical arrow with another.

Hence, the exactness of \eqref{EQ133} and the dual of \eqref{EQ200} together induce  an isomorphism of flat bundles 
\begin{align}
(\mcF_{\mcO_X, p-n}, (D^{\clubsuit \Rightarrow \diamondsuit})^\BBB) \xrightarrow{\sim} (\mcF_{\mcO_X, p-n}^\vee, ((D^{\clubsuit \Rightarrow \diamondsuit})^\BBB)^\vee).
\end{align}
It defines an isomorphism of $(\mr{GL}_{p-n}, o)$-opers $(D^{\clubsuit \Rightarrow \diamondsuit})^\BBB \xrightarrow{\sim} (D^{\clubsuit \Rightarrow \diamondsuit})^{\BBB\BB}$.
This fact implies, under the natural  identification  $\msN_{o} = (\mcO_X, d)$ using the generator $\partial^{p-n-1} \in \mcN_{\mcO_X} \left(= \mcT^{\otimes (p-n-1)} \otimes \mcO_X^{\otimes 2}\right)$,
that  $(D^{\clubsuit \Rightarrow \diamondsuit})^\BBB$ is self-dual.
This completes the proof of this proposition.
\end{proof}

\subsection{Duality between   dormant $\mfs \mfo_{2\ell +1}$- and $\mfs \mfp_{2m}$-opers}\label{SS433}

Let $S$ be a $k$-scheme and $\msX := (f: X \rightarrow S, \{ \sigma_i \}_{i=1}^r)$ an $r$-pointed stable curve of genus $g$ over $S$.
In particular,  $\msX$ associates a log curve $f^\mr{log} : X^\mr{log} \rightarrow S^\mr{log}$.
As before, we fix a dormant $n$-theta characteristic $\vartheta := (\varTheta, \nabla_\vartheta)$ of $X^\mr{log}/S^\mr{log}$ with $1 \leq n < p$.

\bpr \label{Prop458}\begin{itemize}
\item[(i)]
Let $\nabla^\diamondsuit$ be a  self-dual dormant $(\mr{GL}_n, \vartheta)$-oper on $X^\mr{log}/S^\mr{log}$.
Then, the associated dormant $(\mr{GL}_{p-n}, \vartheta^\BBB)$-oper
$\nabla^{\diamondsuit \BBB}$ (cf. \eqref{EQ209}) is self-dual.
Moreover, the assignment $\nabla^\diamondsuit \mapsto \nabla^{\diamondsuit \BBB}$ defines 
a bijection of sets
\begin{align} \label{EQ30}
\begin{pmatrix} \text{the set of isomorphism classes of} \\ \text{self-dual dormant $(\mr{GL}_n, \vartheta)$-opers} \\ \text{on $\msX$}\end{pmatrix}
\xrightarrow{\sim}
\begin{pmatrix} \text{the  set of isomorphism classes of} \\ \text{self-dual dormant $(\mr{GL}_{p-n}, \vartheta^\BBB)$-opers} \\ \text{on $\msX$}\end{pmatrix}.
\end{align}
\item[(ii)]
Let $e := (e_i)_{i=1}^r$ be an $r$-tuple of elements in $\widetilde{\Xi}_{n, \mr{sym}}$.
Then, the bijection   \eqref{EQ30} restricts to a bijection
\begin{align} \label{EQ31}
\begin{pmatrix} \text{the set of isomorphism classes of} \\ \text{self-dual dormant $(\mr{GL}_n, \vartheta)$-opers} \\ \text{of exponents $e$ on $\msX$}\end{pmatrix}
\xrightarrow{\sim}
\begin{pmatrix} \text{the set of isomorphism classes of} \\ \text{self-dual dormant $(\mr{GL}_{p-n}, \vartheta^\BBB)$-opers} \\ \text{of exponents  $e^\BBB$ on $\msX$}\end{pmatrix}.
\end{align}
\end{itemize}
\epr
\begin{proof}
First, we consider assertion (i).
Let us take 
 a scheme-theoretic dense open subscheme  $U$ of $X \setminus \bigcup_{i=1}^r \mr{Im} (\sigma_i)$.
 Then,  the associated morphism 
$U^\mr{log} := U \times_X X^\mr{log} \rightarrow S^\mr{log}$  is strict.
Denote by $\mcO p_{X^\mr{log}/S^\mr{log}}$ (resp., $\mcO p_{U^\mr{log}/S^\mr{log}}$) the \'{e}tale sheaf on $X$ (resp., $U$) that  assigns, to any \'{e}tale scheme $Y$ over $X$ (resp., $U$), the set of isomorphism classes of $(\mr{GL}_n, \vartheta |_Y)$-opers on $(Y \times_X X^\mr{log})/S^\mr{log}$ (resp., $(Y \times_U U^\mr{log})/S^\mr{log}$).
It follows from ~\cite[Theorem 4.49]{Wak5} that 
$\mcO p_{X^\mr{log}/S^\mr{log}}$ and  $\mcO p_{U^\mr{log}/S^\mr{log}}$ carry  torsor structures modeled on ${^\dagger}\mcD_{\leq n -2, \varTheta}$ and ${^\dagger}\mcD_{\leq n -2, \varTheta}|_U$, respectively.
The  morphism of sheaves 
\begin{align} \label{EQ2109}
\mcO p_{U^\mr{log}/S^\mr{log}} \rightarrow \iota_* (\mcO p_{X^\mr{log}/S^\mr{log}}),
\end{align}
where $\iota$ denotes the open immersion $U \hookrightarrow X$,
 obtained via restriction 
is compatible with the  natural morphism ${^\dagger}\mcD_{\leq n -2, \varTheta}|_U \rightarrow \iota_* ({^\dagger}\mcD_{\leq n -2, \varTheta})$.
Since $U$ is scheme-theoretic dense in $X$, the morphism \eqref{EQ2109}   is injective.
Hence, to complete the proof, it suffices to establish  
 the self-duality of $\nabla^{\diamondsuit \BBB}$ after  restriction to $U$.
 In particular, 
we can impose the assumptions in Section \ref{SS231} without loss of generality.
The surjectivity of \eqref{EQ131} implies that  there exists an $(n, \vartheta)$-projective connection $D^{\clubsuit}$ on $U^\mr{log}/S^\mr{log}$ with
$D^{\clubsuit \Rightarrow \diamondsuit} \cong \nabla^\diamondsuit |_U$.
Since $\nabla^\diamondsuit |_U$ is self-dual by assumption, it follows from Proposition \ref{Lemma3}
that  $\nabla^{\diamondsuit \BBB} |_U \left( \cong (\nabla^\diamondsuit |_U)^\BBB \cong (D^{\clubsuit \Rightarrow \diamondsuit})^\BBB\right)$ is self-dual.
This completes the proof of assertion (i).

Assertion (ii) follows from assertion (i) and the bijectivity of \eqref{EQ116}.
\end{proof}

The above assertion immediately yields  the following two assertions.

\bt \label{Prop481e}
Let $\ell$, $m$ be positive integers with $p - 1  = 2 (\ell + m)$.
Then, there exists  a bijection 
\begin{align} \label{EQ234e}
\begin{pmatrix} \text{the set of isomorphism classes of} \\ \text{dormant $(\mr{GO}_{2\ell +1}, \vartheta)$-opers on $\msX$}\end{pmatrix}
\xrightarrow{\sim}
\begin{pmatrix} \text{the set of isomorphism classes of} \\ \text{dormant $(\mr{GSp}_{2m}, \vartheta^\BBB)$-opers on $\msX$}\end{pmatrix}.
\end{align}
If, moreover, $e$ is an $r$-tuple of elements in $\widetilde{\Xi}_{n, \mr{sym}}$ (where we set $e := \emptyset$ when $r =0$),
then this bijection restricts to a bijection 
\begin{align} \label{EQ235e}
\begin{pmatrix} \text{the set of isomorphism classes of} \\ \text{dormant $(\mr{GO}_{2\ell +1}, \vartheta)$-opers} \\ \text{of exponents  $e$ on $\msX$}\end{pmatrix}
\xrightarrow{\sim}
\begin{pmatrix} \text{the set of isomorphism classes of} \\ \text{dormant $(\mr{GSp}_{2m}, \vartheta^\BBB)$-opers} \\ \text{of exponents $e^\BBB$ on $\msX$}\end{pmatrix}.
\end{align}
Finally, the formations of  these bijections commute with base-changes over $S$-schemes.
\et
\begin{proof}
The bijection asserted in (i) can be obtained by composing the bijections
  \eqref{EQ2} and  \eqref{EQ30}.
Also, 
the bijection in (ii) follows  from this  bijection  together with
Proposition \ref{Prop458}, (ii).
\end{proof}

\bt\label{Prop460e}
There exists  a unique dormant $(\mr{GSp}_{p-1}, \vartheta)$-oper on $\msX$, and 
its exponent  at every marked point coincides with   $e_i (\nabla_\vartheta)^\BBB$.
In particular,  if  $r =0$,  $X$ is a geometrically connected, proper, and smooth curve over $S$, and $\vartheta = \vartheta_\mr{Ray}$ (cf. \eqref{EQ281}), then  this dormant $(\mr{GSp}_{p-1}, \vartheta)$-oper is isomorphic to $\nabla_{\ang, \mr{Ray}}^{\diamondsuit}$ (cf. \eqref{EQ63}).
\et
\begin{proof}
The assertion follows from  Proposition \ref{Prop458} for $n=1$ together with the fact that  there exists a unique (up to isomorphism)  dormant $(\mr{GL}_1, \vartheta)$-oper, i.e., $(\mcF_{\varTheta, 1}, \nabla_\vartheta)$.
\end{proof}

Moreover, by  the bijection  \eqref{EQ5},   
Theorems \eqref{Prop481e} and \eqref{Prop460e} yield  the following Theorems  \ref{Prop461} and \ref{Prop460}, respectively.

\bt \label{Prop461}
Let $\ell$, $m$ be positive integers with $p - 1  = 2 (\ell + m)$.
Then, there exists  a bijection 
\begin{align} \label{EQ234}
\begin{pmatrix} \text{the set of isomorphism classes of} \\ \text{dormant $\mfs \mfo_{2\ell +1}$-opers on $\msX$}\end{pmatrix}
\xrightarrow{\sim}
\begin{pmatrix} \text{the set of isomorphism classes of} \\ \text{dormant $\mfs \mfp_{2m}$-opers on $\msX$}\end{pmatrix}.
\end{align}
If, moreover, $\rho$ is an $r$-tuple of elements in $\Xi_{n, \mr{sym}}$ (where we set  $\rho := \emptyset$ when $r =0$),
then this bijection restricts to a bijection 
\begin{align} \label{EQ235}
\begin{pmatrix} \text{the set of isomorphism classes of} \\ \text{dormant $\mfs \mfo_{2\ell +1}$-opers of radii $\rho$ on $\msX$}\end{pmatrix}
\xrightarrow{\sim}
\begin{pmatrix} \text{the set of isomorphism classes of} \\ \text{dormant $\mfs \mfp_{2m}$-opers of radii $\rho^\BBB$ on $\msX$}\end{pmatrix}.
\end{align}
Finally, the formations of  these bijections commute with base-changes over $S$-schemes.
\et

\bt\label{Prop460}
There exists  a unique dormant $\mfs \mfp_{p-1}$-oper on $\msX$, and 
its radius at every marked point coincides with  $\rho_\mr{full} := \pi (\mbF_p)$ (cf. \eqref{EQ301}).
In particular,  when $r =0$ and $X$ is a geometrically connected,  proper, and  smooth curve over $S$, this dormant $\mfs \mfp_{p-1}$-oper  is isomorphic to $\nabla_{\ang, \mr{Ray}}^{\diamondsuit \Rightarrow \spadesuit}$  (cf. \eqref{EQ63}).
\et

\subsection{Moduli spaces of dormant opers}\label{SS155}

We set  $\mfg := \mfs \mfl_n$  (resp., $\mfg := \mfs \mfo_{2\ell +1}$; resp., $\mfg := \mfs \mfp_{2m}$) for some integer $n$ (resp., $\ell$; resp., $m$) with $1 < n < p$  (resp., $1 < 2\ell +1 < p$; resp., $1 < 2m < p$),
 and write 
 \begin{align} \label{EQ300}
 \Xi_{\mfg} := \Xi_n \  \left(\text{resp.,} \  \Xi_{\mfg} := \Xi_{2\ell +1, \mr{sym}}; \text{resp.,}  \ \Xi_{\mfg} := \Xi_{2m, \mr{sym}}\right).
 \end{align}

We denote by 
\begin{align} \label{EQ27}
\mcO p^{^\mr{Zzz...}}_{\mfg, g, r}
\end{align}
the category consisting of   pairs $(\msX, \msE^\spadesuit)$, where $\msX$ denotes an $r$-pointed stable curve  of genus $g$ over a $k$-scheme   and $\msE^\spadesuit$ denotes  a dormant $\mfg$-oper   on $\msX$.
The assignment from such a pair $(\msX, \msE^\spadesuit)$ to the base scheme of $\msX$ defines a functor $\mcO p^{^\mr{Zzz...}}_{\mfg, g, r} \rightarrow \mcS ch_k$,
where $\mcS ch_k$ denotes the category of $k$-schemes (equipped with the big \'{e}tale topology).
With respect to  this functor, 
$\mcO p^{^\mr{Zzz...}}_{\mfg, g, r}$ forms a stack over $\mcS ch_k$.

Moreover,  for each $r$-tuple 
 $\rho := (\rho_i)_{i=1}^r$ of   elements in  $\Xi_{\mfg}$ (where $\rho := \emptyset$ if  $r =0$),
 we obtain a (possibly empty) substack
 \begin{align} \label{EQ28}
 \mcO p^{^\mr{Zzz...}}_{\mfg, \rho, g, r}
 \end{align}
 of $\mcO p^{^\mr{Zzz...}}_{\mfg, g, r}$ classifying dormant $\mfg$-opers of radii $\rho$.
 We know that  $\mcO p^{^\mr{Zzz...}}_{\mfg, g, r}$ decomposes as the disjoint union $\mcO p^{^\mr{Zzz...}}_{\mfg, g, r} = \coprod_{\rho \in \Xi_{\mfg}^r}  \mcO p^{^\mr{Zzz...}}_{\mfg, \rho, g, r}$.
The assignment $(\msX, \msE^\spadesuit) \mapsto \msX$ defines a morphism of stacks
 $\Pi_{\mfg, \rho, g, r} : \mcO p^{^\mr{Zzz...}}_{\mfg,  \rho, g, r} \rightarrow \overline{\mcM}_{g, r}$.

According to  ~\cite[Theorem 4.66 and Proposition 4.71]{Wak5},  $\mcO p_{\mfs \mfl_n, \rho, g, r}^{^\mr{Zzz...}}$ can be represented by a Deligne-Mumford stack over $k$.
It follows from  ~\cite[Theorems 3.33, (i), and 8.19]{Wak5} together with ~\cite[Theorems A and B]{Wak5} that
$\Pi_{\mfs \mfl_n, \rho, g, r}$ is finite and generically \'{e}tale.
By  the functorial bijection \eqref{EQ110}, the assignment $\nabla^\diamondsuit \mapsto \nabla^{\diamondsuit \BBB}$ defines, via \eqref{EQ33}, an isomorphism of $\overline{\mcM}_{g, r}$-stacks
\begin{align} \label{EQ229}
\mcO p^{^\mr{Zzz...}}_{\mfs \mfl_n, \rho, g, r} \xrightarrow{\sim}\mcO p^{^\mr{Zzz...}}_{\mfs \mfl_{p-n}, \rho^\BBB, g, r}. 
\end{align}
Composing this with the isomorphism $\mcO p_{\mfs \mfl_{p-n}, \rho^\BBB, g, r}^{^\mr{Zzz...}} \xrightarrow{\sim} \mcO p_{\mfs \mfl_{p-n}, \rho^\BBBB, g, r}^{^\mr{Zzz...}}$ induced by \eqref{EQ601}, we obtain an isomorphism of $\overline{\mcM}_{g, r}$-stacks
\begin{align}\label{EQ604}
\Dual_{n, \rho, g, r} : \mcO p_{\mfs \mfl_n, \rho, g, r}^{^\mr{Zzz...}} \xrightarrow{\sim} \mcO p_{\mfs \mfl_{p-n}, \rho^\BBBB, g, r}^{^\mr{Zzz...}}.
\end{align}
This isomorphism  satisfies
$\Dual_{p-n, \rho^\BBBB, g, r} \circ  \Dual_{n, \rho, g, r} = \mr{id}$ (cf. ~\cite[Theorem A]{Wak2}).

In the case $n= p-1$,  the projection $\Pi_{\mfs \mfl_{p-1}, \rho_\mr{full}^r, g, r}$ for  $\rho_\mr{full}^r := (\rho_\mr{full}, \cdots, \rho_\mr{full}) \in \Xi_{\mfs \mfl_{p-1}}^r$  defines an isomorphism
\begin{align}
\left(\mcO p^{^\mr{Zzz...}}_{\mfs \mfl_{p-1}, g, r} =\right) \mcO p^{^\mr{Zzz...}}_{\mfs \mfl_{p-1}, \rho_\mr{full}^r, g, r} \xrightarrow{\sim}\overline{\mcM}_{g, r}
\end{align}
 (cf. ~\cite[Theorem B]{Wak2}).

Next, under the assumption that $n = 2\ell +1$ (resp., $n=2m$),  it follows from Corollary \ref{Prop89} that 
the stack  $\mcO p_{\mfs \mfo_{2\ell +1}, \rho, g, r}^{^\mr{Zzz...}}$ (resp., $\mcO p_{\mfs \mfp_{2m}, \rho, g, r}^{^\mr{Zzz...}}$) for $\rho \in \Xi_{\mfs \mfo_{2\ell +1}}^r \subseteq \Xi_{\mfs \mfl_n}^r$ (resp., $\rho \in \Xi_{\mfs \mfp_{2m}}^r\subseteq \Xi_{\mfs \mfl_n}^r$)  can be realized as a closed substack of $\mcO p_{\mfs \mfl_n, \rho, g, r}^{^\mr{Zzz...}}$.
In particular,  $\mcO p_{\mfs \mfo_{2\ell +1}, \rho, g, r}^{^\mr{Zzz...}}$ (resp., $\mcO p_{\mfs \mfp_{2m}, \rho, g, r}^{^\mr{Zzz...}}$) are finite over $\overline{\mcM}_{g, r}$.
We are thus led to  the following theorem, which constitutes  the main result of this paper.

\bt[cf. Theorem \ref{ThmA}] \label{QW1199}
\begin{itemize}
\item[(i)]
Suppose that  $p > n = 2\ell +1$ (resp., $p-1> n=2m$).
Let 
 $\rho$ be an element of $\Xi_{\mfs \mfo_{2\ell +1}}^r$ (resp., $\Xi_{\mfs \mfp_{2m}}^r$), which implies  that $\rho^\BBBB$ lies in $\Xi_{\mfs \mfp_{2m}}^r$ (resp.,  $\Xi_{\mfs \mfo_{2\ell +1}}^r$).
Then, 
the isomorphism \eqref{EQ604} restricts to an
 isomorphism of $\overline{\mcM}_{g, r}$-stacks
 \begin{align} \label{EQ610}
 \Dual^\ang_{n, \rho, g, r} : \mcO p_{\mfs \mfo_{n} \rho, g, r}^{^\mr{Zzz...}} \xrightarrow{\sim} 
 \mcO p_{\mfs \mfp_{p-n}, \rho^\BBBB, g, r}^{^\mr{Zzz...}} \ \left(\text{resp.,} \  \Dual^\ang_{n, \rho, g, r} : \mcO p_{\mfs \mfp_{n} \rho, g, r}^{^\mr{Zzz...}} \xrightarrow{\sim} 
 \mcO p_{\mfs \mfo_{p-n}, \rho^\BBBB, g, r}^{^\mr{Zzz...}} \right).
 \end{align}
In particular, the equality $\Dual_{p-n, \rho^\BBBB, g, r}^\ang \circ \Dual^\ang_{n, \rho, g, r} = \mr{id}$ holds.
 \item[(ii)]
 The projection $\Pi_{\mfs \mfo_{p-1}, \rho_\mr{full}^r, g, r}$ defines an isomorphism
 \begin{align} \label{EQ284}
 \left(\mcO p^{^\mr{Zzz...}}_{\mfs \mfp_{p-1}, g, r} =\right) \mcO p^{^\mr{Zzz...}}_{\mfs \mfp_{p-1}, \rho_\mr{full}^r, g, r} \xrightarrow{\sim}\overline{\mcM}_{g, r}.
 \end{align}
 \end{itemize}
\et
\begin{proof}
Assertions (i) and (ii) follow from Theorems \ref{Prop461} and  \ref{Prop460}, respectively.
\end{proof}

\bco[cf. Theorem \ref{ThmB}, (i)] \label{Cor21}
Suppose that $\mfg$ is either $\mfs \mfo_{2\ell +1}$ with $1 \leq \ell \leq \frac{p-3}{2}$, or $\mfs \mfp_{2m}$ with $1 \leq m \leq \frac{p-3}{2}$.
For each  $\rho := (\rho_i)_{i=1}^r \in \Xi_\mfg^r$ (where  we set $\rho := \emptyset$ if $r =0$),
the natural projection $\Pi_{\mfg, \rho, g, r}$ is generically \'{e}tale.
\eco
\begin{proof}
The assertion for   $\mfg = \mfs \mfo_{2\ell +1}$ with $4\ell +2 < p$ and $\mfg = \mfs \mfp_{2m}$ with  $4m <p$ was established in  ~\cite[Theorem 8.19]{Wak5}.
The remaining cases are reduced to these situations by means of  Theorem \ref{QW1199}, (i).
\end{proof}

\subsection{Factorization property of the number of dormant opers} \label{SS7}

We conclude this paper by describing factorization properties,   with respect to  the data of radii, of the generic degrees $\mr{deg}(\Pi_{\mfg, \rho,  g, r})$ of  $\Pi_{\mfg, \rho, g, r}$ for  $\mfg = \mfs \mfo_{2\ell +1}, \mfs \mfp_{2m}$ in higher-rank cases
  (cf. ~\cite[Chap.\,7]{Wak5} for the previous study of related topics).
These factorization phenomena are naturally governed by
pseudo-fusion rules,
and may be viewed as higher-rank extensions of the results obtained in
earlier work (cf.~\cite[Section 7]{Wak5}).

To begin with, let us consider the case 
$\mfg = \mfs \mfo_{2\ell +1}$ for some positive integer $\ell$ with $1 \leq \ell \leq \frac{p-3}{2}$.
Assume first that   $4 \ell + 2 < p$.
Since each element of $\Xi_{\mfs \mfo_{2\ell +1}}$ is invariant under $(-)^\BB$,
it follows from ~\cite[Theorem G]{Wak5} that  $\mfs \mfo_{2\ell +1}$ satisfies 
both  conditions $(*)$ and $(**)$ introduced  at the beginning of ~\cite[Section 7.3.5]{Wak5}.
Thus, by  ~\cite[Proposition 7.33]{Wak5} together with  the discussion in 
~\cite[Section 7.4]{Wak5},
one can obtain  the {\it pseudo-fusion ring}  for dormant $\mfs \mfo_{2\ell +1}$-opers 
\begin{align}
\Fus_{\mfs \mfo_{2\ell +1}},
\end{align}
 in the sense of ~\cite[Definition 7.34]{Wak5}.
To be precise, $\Fus_{\mfs \mfo_{2\ell +1}}$ is defined as the unitization of the free abelian group $\mbZ^{\Xi_{\mfs \mfo_{2\ell +1}}}$ with basis $\Xi_{\mfs \mfo_{2\ell +1}}$ equipped with the multiplication $\ast : \mbZ^{\Xi_{\mfs \mfo_{2\ell +1}}} \times \mbZ^{\Xi_{\mfs \mfo_{2\ell +1}}} \rightarrow \mbZ^{\Xi_{\mfs \mfo_{2\ell +1}}}$ given by
\begin{align} \label{EQ49}
\alpha \ast \beta = \sum_{\lambda \in \Xi_{\mfs \mfo_{2\ell +1}}} \mr{deg}(\Pi_{\mfs \mfo_{2\ell +1}, (\alpha, \beta, \lambda), 0, 3}) \cdot \lambda.
\end{align}

On the other hand, assume  that $4\ell +2 >p$.
When we set   $m := \frac{p-1-2\ell}{2}$ $\left(\Longleftrightarrow p-1 = 2 (\ell +m) \right)$,
 this assumption is equivalent  to the inequality
 $4m < p$.
 Hence, we can 
apply again  the same results in ~\cite{Wak5} as above, and  thereby obtain the pseudo-fusion ring for dormant $\mfs \mfp_{2m}$-opers $\Fus_{\mfs \mfp_{2m}}$.
Under the identification $\Xi_{\mfs \mfo_{2\ell +1}} = \Xi_{\mfs \mfp_{2m}}$ given by the involution $(-)^\BBB$,
the multiplication structure on  $\Fus_{\mfs \mfp_{2m}}$ is transposed to 
 $\mbZ^{\Xi_{\mfs \mfo_{2\ell +1}}}$;
we denote  the resulting ring, whose underlying abelian group is $\mbZ^{\Xi_{\mfs \mfo_{2\ell +1}}}$,   by
$\Fus_{\mfs \mfo_{2\ell +1}}$.
Since  $\mr{deg}(\Pi_{\mfs \mfo_{2\ell +1}, \rho, g, r}) = \mr{deg} (\Pi_{\mfs \mfp_{2m}, \rho^\BBB, g, r})$ for every collection $(\rho, g, r)$,
the multiplication in $\Fus_{\mfs \mfo_{2\ell +1}}$ is given by the same formula  as in  \eqref{EQ49}.
In particular, 
 even in the case  $4\ell +2 >p$,
we can apply the discussion of  ~\cite[Section 7]{Wak5} to this  ring $\Fus_{\mfs \mfo_{2\ell +1}}$.
The explicit description  of its ring structure enables us to compute 
 the values $\mr{deg}(\Pi_{2\ell +1, \rho, g, r})$.
 
 Furthermore, by reversing the roles of $\mfs \mfo_{2\ell +1}$ and $\mfs \mfp_{2m}$ and reapplying the above argument, we  also obtain 
  the  pseudo-fusion ring  for dormant $\mfs \mfp_{2m}$-opers 
\begin{align}
\Fus_{\mfs \mfo_{2m}}
\end{align}
 for any positive integer $m$ with $1 \leq m \leq \frac{p-3}{2}$.
Thus, we obtain  the following  assertion.

\SSP
\bt[cf.  Theorem \ref{ThmB}, (ii)] \label{Thm15}
Suppose that $\mfg$ is either $\mfs \mfo_{2\ell +1}$ with $1 \leq \ell \leq \frac{p-3}{2}$ or $\mfs \mfp_{2m}$ with $1 \leq m \leq \frac{p-3}{2}$.
Write $\mfS$ for the set of ring homomorphisms $\Fus_\mfg \rightarrow \mbC$ and write $\mr{Cas} := \sum_{\lambda \in \Xi_{\mfg}} \lambda \ast \lambda \left(\in \Fus_\mfg \right)$.
Then, for each $\rho := (\rho_i)_{i=1}^r \in \Xi_\mfg^{\times r}$,
the following equality holds:
\begin{align}
\mr{deg}(\Pi_{\mfg, \rho, g, r}) = \sum_{\chi \in \mfS} \chi (\mr{Cas})^{g-1} \cdot \prod_{i=1}^r \chi (\rho_i).
\end{align}
In particular, if $r = 0$ (which implies $g > 1$), then this equality reads
\begin{align}
\mr{deg}(\Pi_{\mfg,  \emptyset, g, 0}) = \sum_{\chi \in \mfS} \chi (\mr{Cas})^{g-1}.
\end{align}
\et
\begin{proof}
The assertion follows from ~\cite[Theorem 7.36, (ii)]{Wak5}.
\end{proof}

\subsection*{Acknowledgements}
We are grateful for the many constructive conversations we had with
the moduli spaces of dormant $\mfs \mfo_{2\ell +1}$- and $\mfs \mfp_{2m}$-opers, who live in the world of mathematics!
The author was partially supported by 
 JSPS KAKENHI Grant Number 25K06933.


\end{document}